%% file: completion.tex
\documentclass[a4paper]{amsart}
\usepackage{amsmath,amssymb}
\usepackage{stmaryrd}
\usepackage{tikz}
\usetikzlibrary{arrows}
\usepackage{hyperref}

\usepackage{color}

\title[Completion of Choice]
{Completion of Choice}

\author[V.\ Brattka]{Vasco Brattka}
\address{Faculty of Computer Science, Universit\"at der Bundeswehr M\"unchen, Germany and 
             Department of Mathematics \& Applied Mathematics, University of Cape Town, South Africa\footnote{Vasco Brattka has received funding from the National Research Foundation of South Africa}} 
\email{Vasco.Brattka@cca-net.de}

\author[G.\ Gherardi]{Guido Gherardi}
\address{Dipartimento di Filosofia e Comunicazione, Universit\`{a} di Bologna, Italy}
\email{Guido.Gherardi@unibo.it}

\def\AA{{\mathcal A}}

\def\CC{{\mathcal C}}

\def\IN{{\mathbb{N}}}

\def\IR{{\mathbb{R}}}
\def\IS{{\mathbb{S}}}

\def\TO{\Longrightarrow}
\def\In{\subseteq}

\def\prefix{\sqsubseteq}

\def\mto{\rightrightarrows}

\def\id{{\rm id}}

\def\dom{{\rm dom}}
\def\range{{\rm range}}

\def\t{\mathrm t}
\def\r{\mathrm r}

\def\LPO{\text{\rm\sffamily LPO}}
\def\LLPO{\text{\rm\sffamily LLPO}}
\def\WKL{\text{\rm\sffamily WKL}}
\def\RCA{\text{\rm\sffamily RCA}}

\def\BWT{\text{\rm\sffamily BWT}}

\def\B{\text{\rm\sffamily B}}

\def\C{\mbox{\rm\sffamily C}}
\def\ConC{\mbox{\rm\sffamily CC}}

\def\LPO{\mbox{\rm\sffamily LPO}}
\def\LLPO{\mbox{\rm\sffamily LLPO}}

\def\WBWT{\text{\rm\sffamily WBWT}}

\def\K{\text{\rm\sffamily K}}

\def\Low{\text{\rm\sffamily L}}

\def\J{\text{\rm\sffamily J}}
\def\T{\text{\rm\sffamily T}}

\def\ACC{\text{\rm\sffamily ACC}}

\def\MLR{\text{\rm\sffamily MLR}}
\def\WWKL{\text{\rm\sffamily WWKL}}

\def\PC{\text{\rm\sffamily PC}}

\def\PA{\text{\rm\sffamily PA}}

\def\PCC{\text{\rm\sffamily PCC}}

\def\SORT{\text{\rm\sffamily SORT}}
\def\WFT{\text{\rm\sffamily WFT}}
\def\INF{\text{\rm\sffamily INF}}

\def\NEG{\text{\rm\sffamily NEG}}
\def\ATR{\text{\rm\sffamily ATR}}

\def\leqW{\mathop{\leq_{\mathrm{W}}}}
\def\leqTW{\mathop{\leq_{\mathrm{tW}}}}
\def\equivW{\mathop{\equiv_{\mathrm{W}}}}

\def\leqSW{\mathop{\leq_{\mathrm{sW}}}}
\def\leqSTW{\mathop{\leq_{\mathrm{stW}}}}

\def\equivSW{\mathop{\equiv_{\mathrm{sW}}}}
\def\equivSTW{\mathop{\equiv_{\mathrm{stW}}}}

\def\nleqW{\mathop{\not\leq_{\mathrm{W}}}}

\def\nleqSW{\mathop{\not\leq_{\mathrm{sW}}}}

\def\lW{\mathop{<_{\mathrm{W}}}}
\def\lTW{\mathop{<_{\mathrm{tW}}}}
\def\lSW{\mathop{<_{\mathrm{sW}}}}

\def\nW{\mathop{|_{\mathrm{W}}}}

\date{\today}

\newtheorem{theorem}{Theorem}[section]
\newtheorem{proposition}[theorem]{Proposition}
\newtheorem{lemma}[theorem]{Lemma}

\newtheorem{corollary}[theorem]{Corollary}

\theoremstyle{definition}
\newtheorem{definition}[theorem]{Definition}

\newtheorem{example}[theorem]{Example}

\newtheorem{question}[theorem]{Question}

\begin{document}

\begin{abstract}
We systematically study the completion of choice problems in the Weihrauch lattice. 
Choice problems play a pivotal r\^{o}le in Weihrauch complexity.
For one, they can be used as landmarks that characterize important equivalences
classes in the Weihrauch lattice. 
On the other hand, choice problems also characterize several natural classes of computable problems, such as finite mind change computable problems, 
non-deterministically computable problems, Las Vegas computable problems
and effectively Borel measurable functions.
The closure operator of completion generates the concept of total
Weihrauch reducibility, which is a variant of Weihrauch reducibility with
total realizers.  Logically speaking, the completion of a problem is 
a version of the problem that is independent of its premise.
Hence, studying the completion of choice problems allows us to study simultaneously choice problems in the total Weihrauch lattice,
as well as the question which choice problems can be made independent
of their premises in the usual Weihrauch lattice.
The outcome shows that many important choice problems that are related to 
compact spaces are complete, whereas choice problems for 
unbounded spaces or closed sets of positive measure are typically not complete.
\end{abstract}

\maketitle

%\begin{footnotesize}
\setcounter{tocdepth}{1}
\tableofcontents
%\end{footnotesize}

\pagebreak

\section{Introduction}
\label{sec:introduction}

Choice problems play a crucial r\^{o}le in Weihrauch complexity.
A recent survey on the field can be found in~\cite{BGP18}.
A choice problem is a problem of the logical form
\[(\forall\mbox{ closed }A\In X)(A\in D\TO(\exists x\in X)\,x\in A).\]
Here $X$ is typically a computable metric space, the closed set $A\In X$ is typically given
by negative information in order to make the statement non-trivial, and the premise $D$ could be a property such
as non-emptiness, sometimes combined with additional properties, such as 
positive measure, connectedness, etc. The multi-valued Skolem function
of such a choice problem is a function of the form
\[\C_X:\In\AA_-(X)\mto X,A\mapsto A,\]
where $\AA_-(X)$ denotes the the space of closed subsets of $X$ with respect to negative information
and $\dom(\C_X)=\{A:A\not=\emptyset\}$ is a particular $D$.
If the domain is further restricted to sets of positive measure or connected sets,
then we denote the problem by $\PC_X$ and $\ConC_X$, respectively. 
By $\K_X$ we denote compact choice that considers compact sets with respect to negative information.
Many basic systems from reverse mathematics~\cite{Sim09} have
certain choice problems as uniform counterparts. 
Also some classes of problems that are computable in a certain sense can be characterized
as cones below certain choice problems in the Weihrauch lattice.
The Table~\ref{fig:reverse} gives a survey on such correspondences
(see \cite{BGP18} for further details).

\begin{table}[htb]
\begin{center}
\begin{tabular}{lll}
{choice problems\ } & {reverse mathematics} & {classes of problems}\\\hline
$\C_1$  & $\RCA_0^*$ & computable\\
$\K_\IN$ & ${\mathsf B}\Sigma^0_1$ & \\
$\C_\IN$  & ${\mathsf I}\Sigma^0_1$ & finite mind change computable\\
$\C_{2^\IN}$  & $\WKL^*$& non-deterministically computable\\
$\PC_{2^\IN}$  & $\WWKL^*$ & Las Vegas computable\\
$\C_{\IN^\IN}$ & $\ATR_0$& effectively Borel measurable functions\\[0.3cm]
\end{tabular}
\end{center}
\caption{Weihrauch complexity and reverse mathematics}
\label{fig:reverse}
\end{table}

We assume that the reader is familiar with Weihrauch reducibility $\leqW$.
The statement $f\leqW g$ roughly speaking expresses that the problem $f$ can be computably reduced to the problem
$g$ in the sense that each realizer of $g$ computes a realizer of $f$ in a uniform way (see \cite{BGP18}).
In \cite{BG20} we have introduced the closure operator of completion $f\mapsto\overline{f}$ that induces total Weihrauch reducibility $\leqTW$ by
\[f\leqTW g:\iff f\leqW\overline{g}.\]
Total Weihrauch reducibility $\leqTW$ is a variant of the usual concept
of Weihrauch reducibility $\leqW$ and it can be directly defined using
total realizers~\cite{BG20}. Our main motivation for studying this total variant of Weihrauch reducibility
and the completion operator $f\mapsto\overline{f}$ is that one can obtain a Brouwer algebra in this way.
More precisely, if the completion operator is combined with the closure operator $f\mapsto\widehat{f}$ of parallelization,
then the resulting lattice structure is a Brouwer algebra, i.e., a model of some intermediate logic that, like in the case of the
Medvedev lattice, turns out to be Jankov logic~\cite{BG20}.

Formally, the completion $\overline{f}:\overline{X}\mto\overline{Y}$
of a problem $f:\In X\mto Y$ is defined by
\[\overline{f}(x):=\left\{\begin{array}{ll}
   f(x) & \mbox{if $x\in\dom(f)$}\\
   \overline{Y} & \mbox{otherwise}
\end{array}\right.,\]
i.e., by a totalization of $f$ on the completions $\overline{X},\overline{Y}$ of the corresponding types.\footnote{The completion of types has interesting independent applications and was recently introduced and used by Dzhafarov~\cite{Dzh19} to show that a strong variant $\leqSW$ of Weihrauch reducibility
actually forms a lattice structure.}
Logically, the completion $\overline{f}$ of a problem $f$ can be seen as a way to make $f$ independent
of its premise. For choice problems this means to consider them in the form
\[(\forall A\in\overline{\AA_-(X)})(\exists x\in\overline{X})(A\in D\TO x\in A),\]
where the existence of $x$ is now independent of the premise $A\in D$.
If we use intuitionistic logic, then we cannot just export the quantifier without 
changing the meaning of the formula. Likewise, the computational content
of the formula with the exported quantifier is different from the original one.

\begin{figure}[tb]
\begin{center}
\input{BasicChoice}
\end{center}
\ \\[-0.5cm]
\caption{Basic problems and their completions in the Weihrauch lattice}
\label{fig:choice}
\end{figure}
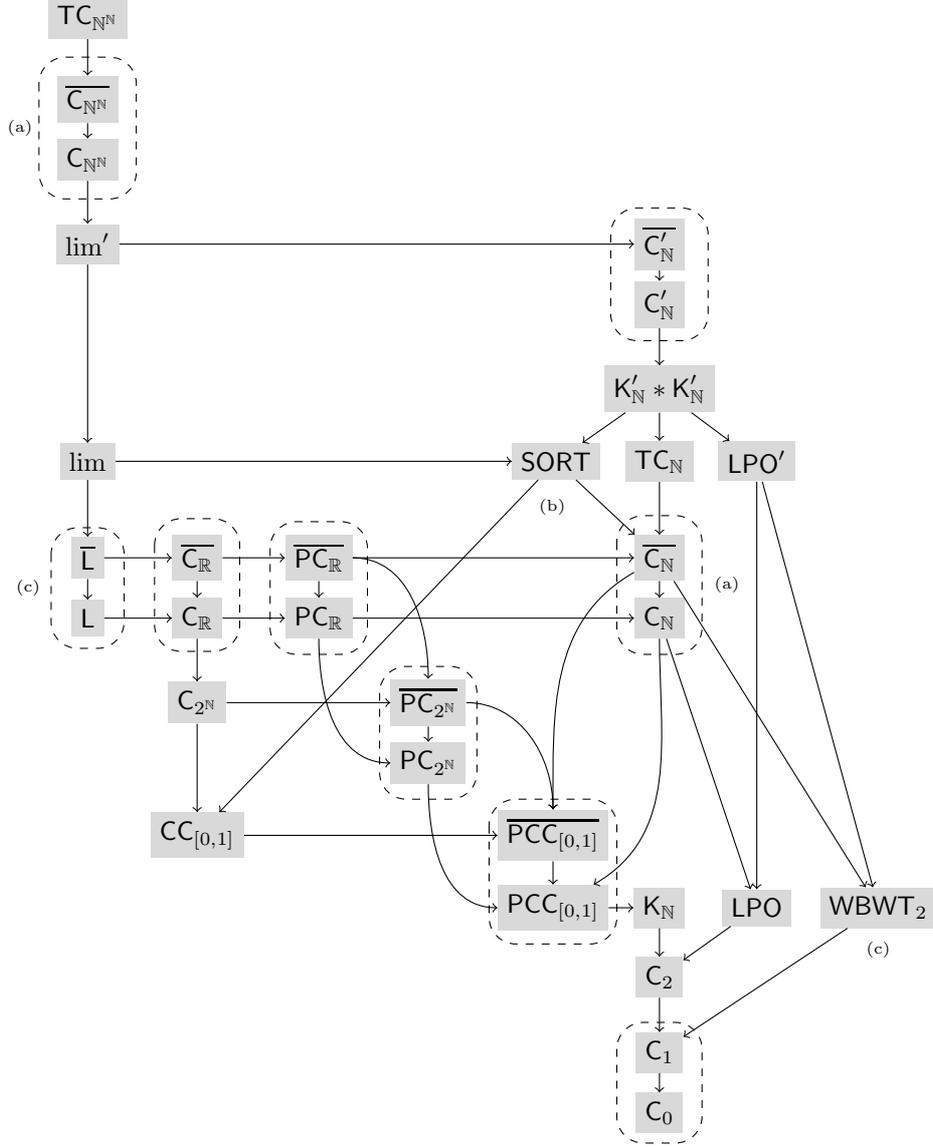

The main question that we study in this article is: which choice problems
$\C_X$ and their variants are complete, i.e., when do we obtain $\overline{\C_X}\equivW\C_X$?

Some examples of complete and incomplete choice problems are the following:

\begin{itemize}
\item Complete choice problems: $\C_n$ for $n\geq1$, $\K_\IN$, $\C_{2^\IN}$, $\ConC_{[0,1]}$.
\item Incomplete choice problems: $\C_0,\C_\IN,\C_\IR,\C_{\IN^\IN},\PC_{2^\IN},\PCC_{[0,1]}$.
\end{itemize}

From the perspective of a complete problem its lower cones in the 
Weihrauch lattice and in the total Weihrauch lattice coincide.
Together with the notion of completeness we also study the notion
of co-completeness. For co-complete problems the upper cones
in the two lattices coincide. Since many important problems are
either complete or co-complete or even both, we obtain
very similar reducibility relations between important choice problems
in the usual and the total Weihrauch lattice. The incomplete problems in Figure~\ref{fig:choice} are all
shown in dashed boxes together with their completions. If we disregard the completions, then
all relations between any two problems shown in Figure~\ref{fig:choice} are the same for ordinary
Weihrauch reducibility $\leqW$ and its total variant $\leqTW$, except those that involve the
weak Bolzano-Weierstra\ss{} theorem $\WBWT_2$ on the space $\{0,1\}$.

However, a certain amount of expressiveness is lost by 
the transition from the ordinary Weihrauch lattice to its total variant:

\begin{enumerate}
\item Finite mind change computable problems and 
         Las Vegas computable problems can be characterized
         as lower cones with Weihrauch reducibility $\leqW$, but not with
         total Weihrauch reducibility $\leqTW$.
\item Low problems can be characterized as lower cones with strong
        Weihrauch reducibility $\leqSW$, but not with strong total Weihrauch
        reducibility $\leqSTW$.
\item Limit computable problems and non-deterministically computable problems
        can be characterized as lower cones for all mentioned reducibilities. 
\end{enumerate}

We provide a list of some references for some crucial reductions and separations given in Figure~\ref{fig:choice}.
Several further references can be found in the survey \cite{BGP18}.

\begin{enumerate}
\item[(a)] In Corollary~\ref{cor:completion-choice} we prove that in general
$\C_X\leqSW\overline{\C_X}\leqSW\T\C_X\equivSW\overline{\T\C_X}$ holds.
In Theorem~\ref{thm:choice-Baire} we provide the necessary separations for $X=\IN^\IN$
and in Corollary~\ref{cor:choice-natural} the corresponding separations for $X=\IN$.
The reduction $\lim'\leqW\C_{\IN^\IN}$ follows, for instance, from \cite[Theorem~7.7]{BBP12}.
The reduction $\T\C_\IN\leqW\K_\IN'*\K_\IN'$ follows from Corollary~\ref{cor:CN-!CN} and Proposition~\ref{prop:KN-CN}.
\item[(b)] Neumann and Pauly introduced $\SORT$ and proved $\C_\IN\lW\SORT\lW\lim$ \cite[Proposition~24]{NP18}.
This is improved by Corollary~\ref{cor:CN-SORT}, which yields $\overline{\C_\IN}\leqW\SORT$.
The reduction
$\SORT\leqW\K_\IN'*\K_\IN'$ follows from Corollary~\ref{cor:CN-!CN} and Proposition~\ref{prop:KN-CN}.
The reduction $\ConC_{[0,1]}\leqW\SORT$ was proved in \cite[Proposition~16]{BHK17}.
\item[(c)] The reduction $\C_\IR\leqW\Low$ was proved in \cite[Corollary~4.9, Theorem 8.7]{BBP12}.
The separation of $\overline{\Low}$ and $\Low$ and, in fact, several other separations in the diagram follow,
since $\WBWT_2\nleqW\Low$ by Proposition~\ref{prop:WBWT-L} and $\WBWT_2\leqW\overline{\C_\IN}$
by Proposition~\ref{prop:WBWT}. The problem $\WBWT_2$ was introduced in \cite{BHK18}.
By \cite[Corollary~11.11]{BGM12} we have $\BWT_2\equivSW\LLPO'$ and hence
$\WBWT_2\leqW\BWT_2\leqW\LPO'$.
\end{enumerate}

In the following section~\ref{sec:precomplete} we continue the study of precomplete
representations that was started in \cite{BG20}. We characterize represented spaces
that admit total precomplete representations as spaces that allow computable
multi-valued retractions from their completions onto themselves. 
We call such spaces multi-retraceable. 
In section~\ref{sec:total} we briefly recall some basic facts about total
Weihrauch reducibility that were provided in \cite{BG20}. 
In section~\ref{sec:completion} we continue the study of completion of problems 
that was started in \cite{BG20} and we introduce the notion of co-completeness
and co-totality. In particular, we introduce a criterion that is sufficient to
guarantee co-completeness and co-totality for jumps of non-constant discrete
functions. In section~\ref{sec:choice} we start to study the main theme of this
article, namely the completion of choice problems. We formulate a number
of results that hold for general choice problems and in section~\ref{sec:compact}
we focus on choice on compact spaces. While choice on Cantor space,
on non-empty finite spaces and connected choice on the unit interval
are complete, most other choice principles that we study are incomplete.
In section~\ref{sec:positive-choice} we establish incompleteness of
choice problems for sets of positive measure and in section~\ref{sec:choice-natural}
we establish incompleteness of choice for natural numbers.
In section~\ref{sec:lowness} we briefly discuss lowness and we show
that the low problem $\Low:=\J^{-1}\circ\lim$ is not complete.
Finally, in section~\ref{sec:choice-Euclidean} we discuss variants of choice
on Euclidean space and in section~\ref{sec:choice-Baire} choice on Baire space.

\section{Precompleteness, Completeness and Retraceability}
\label{sec:precomplete}

We recall that a {\em represented space} $(X,\delta)$ is a set $X$ together with
a surjective (partial) map $\delta:\In\IN^\IN\to X$, called the {\em representation} of $X$.
For the purposes of our topic so-called {\em precomplete representations}
are important. They were introduced by Kreitz and Weihrauch \cite{KW85} following the
concept of a precomplete numbering introduced by Er\v{s}ov~\cite{Ers99}.
We recall some results on precomplete representations from \cite{BG20}.

\begin{definition}[Precompleteness]\rm
A representation $\delta:\In\IN^\IN\to X$ is called {\em precomplete},
if for any computable $F:\In\IN^\IN\to\IN^\IN$ there exists
a total computable $G:\IN^\IN\to\IN^\IN$ such that
$\delta F(p)=\delta G(p)$ for all $p\in\dom(F)$.
\end{definition}

We recall that for two representations $\delta_1,\delta_2$ of the same set $X$
we say that $\delta_1$ is {\em computably reducible} to $\delta_2$, in symbols $\delta_1\leq\delta_2$,
if and only if there is a computable $F:\In\IN^\IN\to\IN^\IN$ such that $\delta_1=\delta_2F$.
We denote the corresponding equivalence by $\equiv$. 

For $p\in\IN^\IN$ we denote by $p-1\in\IN^\IN\cup\IN^*$ the sequence or word that is formed as concatenation of 
$p(0)-1$, $p(1)-1$, $p(2)-1$,... with the understanding that $-1=\varepsilon$ is the empty word.
If  $(X,\delta_X)$ is a represented space, then the {\em precompletion} $\delta_X^\wp$ of
$\delta_X$ is defined by $\delta_X^\wp(p):=\delta_X(p-1)$ for all $p\in\IN^\IN$ such that
$p-1\in\dom(\delta_X)$.
In \cite[Proposition~3.4]{BG20} we proved that $\delta_X^\wp$ is always precomplete and satisfies $\delta_X^\wp\equiv\delta_X$.

There are also many natural examples for precomplete representations, for instance 
it is not hard to see that the
standard representation of a second-countable $T_0$--space is precomplete, if defined appropriately (see~\cite[Lemma~3.4.8~(6)]{Wei87}).

\begin{example}
\label{ex:T0} 
For every second-countable $T_0$--space $X$ with a countable subbase $(U_n)_{n\in\IN}$
we can define a representation $\delta_X:\In\IN^\IN\to X$ by
\[\delta_X(p)=x:\iff\{n\in\IN:x\in U_n\}=\{n\in\IN:n+1\in\range(p)\}.\]
The representation $\delta_X$ is precomplete.
\end{example}

We will also need the fact that other classes of functions can be extended to total
ones under precomplete representations. In \cite{BG20} we have introduced the following concept.

\begin{definition}[Respect for precompleteness]
We say that a set $P$ of functions $F:\In\IN^\IN\to\IN^\IN$ {\em respects precompleteness},
if for every precomplete representation $\delta$ and any function $F\in P$ there
exists a total function $G\in P$ such that $\delta F(p)=\delta G(p)$ for all $p\in\dom(F)$.
\end{definition}

In \cite[Proposition~3.6]{BG20} we proved that the classes of computable, continuous, limit computable, Borel measurable and non-uniformly computable partial functions $F:\In\IN^\IN\to\IN^\IN$ respect precompleteness.
Later in Corollaries~\ref{cor:finite-mind-change} and \ref{cor:lowness} we are going to prove that functions that are computable with finitely many mind changes
and low functions do not respect precompleteness.

If $(X,\delta_X)$ is a represented space, then its completion $(\overline{X},\delta_{\overline{X}})$ 
is defined by $\overline{X}:=X\cup\{\bot\}$, where $\bot\not\in X$ and $\delta_{\overline{X}}(p):=\delta_X^\wp(p)$ if $p\in\dom(\delta_X^\wp)$ and $\delta_{\overline{X}}(p):=\bot$ otherwise.
The concept of completion was introduced by Damir Dzhafarov in \cite{Dzh19} with a slightly different but equivalent construction.
The construction used here was introduced in \cite{BG20}.
We note that strictly speaking $\overline{X}$ does not only depend on $X$, but also on the underlying representation.
That is, the completions with respect to two computably equivalent representations are not necessarily computably equivalent.

By a {\em problem} $f:\In X\mto Y$ we mean a partial multi-valued map $f:\In X\mto Y$
on represented spaces $(X,\delta_X)$ and $(Y,\delta_Y)$. 
We recall that {\em composition} of problems $f:\In X\mto Y$ and $g:\In Y\mto Z$
is defined by 
\[g\circ f(x):=\{z\in Z:(\exists y\in f(x))\; z\in g(y)\}\]
for all $x\in\dom(g\circ f):=\{x\in\dom(f):f(x)\In\dom(g)\}$.
For two problems $f:\In X\mto Y$ and $g:\In X\mto Z$ with identical source space $X$ we define
the {\em juxtaposition} $(f,g):\In X\mto Y\times Z$ by $(f,g)(x):=f(x)\times g(x)$ and $\dom(f,g):=\dom(f)\cap\dom(g)$.
If $f,g:\In \IN^\IN\mto\IN^\IN$ are problems on Baire space, then we also call $\langle f,g\rangle:=\langle\;\rangle\circ(f,g)$ the
{\em juxtaposition} of $f$ and $g$.

We say that a function $F:\In\IN^\IN\to\IN^\IN$ is a {\em realizer} of $f$,
if $\delta_YF(p)\in f\delta_X(p)$ for all $p\in\dom(f\delta_X)$.
We denote this by $F\vdash f$. 
We say that $f$ is {\em computable} if it has a computable realizer.
Other notions, such as continuity, Borel measurability and so forth that are
well-defined for functions $F:\In\IN^\IN\to\IN^\IN$ are transferred in an
analogous manner to problems $f:\In X\mto Y$.

We also need the notion of a (multi-valued) retraction. 
For $Y\In X$ we call $r:X\mto Y$ a {\em retraction} ({\em onto} $Y$), if $r(x)=x$ for all $x\in Y$.
Often retractions are even single-valued. We call a represented space {\em retraceable} if it is a computable retract of its own completion.

\begin{definition}[Retraceability]
A represented space $(X,\delta_X)$ is called {\em multi-}\linebreak
{\em retraceable} if there is a computable retraction $r:\overline{X}\mto X$,
and $(X,\delta_X)$ is called {\em retraceable} if there is a single-valued computable retraction $r:\overline{X}\to X$.
\end{definition}

In \cite[Corollary~3.10]{BG20} we proved that $\delta_{\overline{X}}$ is always precomplete and the injection $\iota:X\to\overline{X}$ is a computable embedding.
We recall that a {\em computable embedding} is a map $f:X\to Y$ that is computable, injective
and has a partial computable inverse. A {\em computable isomorphism} is a computable embedding that is bijective.

\begin{corollary}[Completion]
\label{cor:completion}
For every represented space $(X,\delta_X)$ the completion $\delta_{\overline{X}}$ is a precomplete total representation
and $\iota:X\to\overline{X},x\mapsto x$ is a computable embedding.
\end{corollary}

Sometimes we will have to work with the double completion $\overline{\overline{X}}$, which is not isomorphic
to $\overline{X}$, since it has an extra $\bot$--element. 
However, there is always a computable retraction $r:\overline{\overline{X}}\to\overline{X}$.
In fact, we can prove the following characterizations of multi-retraceable spaces.

\begin{proposition}[Multi-retraceability]
\label{prop:multi-retraceability}
Let $(X,\delta_X)$ be a represented space. Then the following are equivalent:
\begin{enumerate}
\item $X$ admits a precomplete total representation $\delta:\IN^\IN\to X$ with $\delta\equiv\delta_X$.
\item All computable $f:\In\IN^\IN\to X$ have total computable extensions
$g:\IN^\IN\to X$.
\item For all represented spaces $Y$ and all computable $f:\In Y\mto X$ there exists a total
computable $g:Y\mto X$ with $g(y)\In f(y)$ for all $y\in\dom(f)$.
\item $X$ is multi-retraceable, i.e., there is a computable retraction $r:\overline{X}\mto X$.
\end{enumerate}
\end{proposition}
\begin{proof}
``(1)$\TO$(2)'' 
Let $\delta$ be a total and precomplete representation of $X$ with $\delta\equiv\delta_X$.
Let $F:\In\IN^\IN\to\IN^\IN$ be a computable realizer of $f:\In\IN^\IN\to X$.
Then there exists a total computable $G:\IN^\IN\to\IN^\IN$ with 
$\delta F(p)=\delta G(p)$ for all $p\in\dom(F)$, since $\delta$ is precomplete.
Since $\delta$ is total, $G$ is actually a realizer of a function $g:=\delta G:\IN^\IN\to X$ that extends $f$.\\
``(2)$\TO$(3)'' Let $f:\In Y\mto X$ be computable for some represented space $(Y,\delta_Y)$.
Then $f$ has a computable realizer $F:\In\IN^\IN\to\IN^\IN$.
Then $f':=\delta_X\circ F:\In\IN^\IN\to X$ is computable and by (2) it admits
a total computable extension $g':\IN^\IN\to X$. Then
$g:=g'\circ\delta_Y^{-1}:Y\mto X$ is computable and satisfies $g(y)\In f(y)$ for all $y\in\dom(f)$.\\
``(3)$\TO$(4)'' If we apply (3) to the partial computable inverse $\iota^{-1}:\In\overline{X}\to X$
that exists according to Corollary~\ref{cor:completion}, then we obtain the desired 
computable retraction $r:\overline{X}\mto X$.\\
``(4)$\TO$(1)'' 
Let $\delta_{\overline{X}}$ be the total completion of $\delta_X$.
Let $R:\IN^\IN\to\IN^\IN$ be a computable realizer of a retraction $r:\overline{X}\mto X$
and let $S:\In\IN^\IN\to\IN^\IN$ be a computable realizer of the embedding $\iota:X\to\overline{X}$ that
exists according to Corollary~\ref{cor:completion}.
Then $\delta:=\delta_{\overline{X}}\circ S\circ R$ is a total representation of $X$ that extends $\delta_{\overline{X}}|^X$.
Hence $\delta\equiv\delta_{\overline{X}}|^X$.
On the other hand, $\delta_{\overline{X}}|^X\equiv\delta_X$ by Corollary~\ref{cor:completion}.
We still need to prove that $\delta$ is precomplete. 
Let $F:\In\IN^\IN\to\IN^\IN$ be a computable function.
Since $\delta_{\overline{X}}$ is precomplete 
by Corollary~\ref{cor:completion}, it follows that there is a total computable $G:\IN^\IN\to\IN^\IN$
with $\delta_{\overline{X}}\circ S\circ R\circ F(p)=\delta_{\overline{X}}\circ G(p)$ for all
$p\in\dom(S\circ R\circ F)=\dom(F)$.
We note that $S\circ R\circ F(p)\in\dom(\delta_{\overline{X}}|^X)$ and hence 
$G(p)\in\dom(\delta_{\overline{X}}|^X)$.
We obtain $\delta\circ F(p)=\delta_{\overline{X}}\circ G(p)=\delta\circ G(p)$
for all $p\in\dom(F)$, since $\delta$ is an extension of $\delta_{\overline{X}}|^X$. Thus $\delta$ is precomplete.
\end{proof}

This result shows that the notion of multi-retraceability only depends on the equivalence class of $\delta_X$,
unlike the notion of completion $\overline{X}$ which also depends on the explicit underlying representation.
For multi-retraceable spaces $X$ there is not just a computable embedding $\iota:X\to\overline{X}$, but
also a computable retraction $r:\overline{X}\mto X$. Hence, these spaces are closer to being
isomorphic to their own completion than arbitrary represented spaces.

Since by Corollary~\ref{cor:completion}
$\overline{X}$ is always a represented space with a total precomplete representation,
it follows that this space is an example of a retraceable space.

\begin{corollary}[Retraction]
\label{cor:retraction-completion}
For every represented space $X$ there is a computable retraction $r:\overline{\overline{X}}\to\overline{X}$, i.e., $\overline{X}$ is retraceable.
\end{corollary}
\begin{proof}
The representation $\delta_{\overline{X}}$ of $\overline{X}$ is total. Hence the names of $\bot$ with respect to the
representation $\delta_{\overline{\overline{X}}}$ of $\overline{\overline{X}}=\overline{X}\cup\{\bot\}$ are exactly
those names $p\in\IN^\IN$ that contain at most finitely many digits different from zero. 
Hence, a retraction $r:\overline{\overline{X}}\to\overline{X}$ can be realized basically by the map $F:\In\IN^\IN\to\IN^\IN,p\mapsto p-1$,
where the output is filled up with zeros if not enough non-zero content in $p$ becomes available.
This extends $F$ to a total computable map. 
That is, names of the bottom element $\bot\in\overline{\overline{X}}$ are mapped to names of the bottom
element $\bot'\in\overline{X}=X\cup\{\bot'\}$ and otherwise the identity is realized, i.e., $r|_{\overline{X}}=\id_{\overline{X}}$.
Hence, $r$ is a computable retraction.
\end{proof}

In particular,  $\overline{X}$ is actually a retract of $\overline{\overline{X}}$ in the topological sense.

By Proposition~\ref{prop:multi-retraceability} multi-retraceability is a rather strong condition 
and we cannot expect that too many spaces satisfy this condition.
For instance, $\IN^\IN$ is not multi-retraceable, since there are partial computable
$F:\In\IN^\IN\to\IN^\IN$ that do not have total computable extensions. 
This also shows that there are spaces that admit total representations, but no
representation that is total and precomplete simultaneously. 
We assume that every represented space is endowed with the final topology of the representation.
The following result shows that multi-retraceable spaces are necessarily compact (i.e., 
every open cover has a finite subcover; we do not require Hausdorffness).
 
\begin{proposition}[Compactness]
\label{prop:compactness}
Let $(X,\delta_X)$ be a multi-retraceable space. 
Then there is a representation $\delta:2^\IN\to X$ such that $\delta_X\equiv\delta$.
In particular, $X$ is compact.
\end{proposition}
\begin{proof}
Let us assume that $\delta_X$ is precomplete and total.
We consider the computable partial map $F:\In2^\IN\to\IN^\IN$ defined by 
\[F(0^{k_0}1^{n_0+1}0^{k_1+1}1^{n_1+1}0^{k_2+1}...):=n_0n_1n_2...\]
for all $n_i,k_i\in\IN$ with $i\in\IN$.
Since $\delta_X$ is precomplete, there is a total computable $G:\IN^\IN\to\IN^\IN$
with $\delta_XF(p)=\delta_XG(p)$ for all $p\in\dom(F)$.
We let $\delta:=\delta_XG|_{2^\IN}$. This defines a total representation $\delta:2^\IN\to X$,
since $\delta_X$ is total, and $\delta\leq\delta_X$ holds by definition.
For the inverse direction we consider the computable map $H:\IN^\IN\to2^\IN,p\mapsto1^{p(0)+1}01^{p(1)+1}01^{p(2)+1}...$,
which satisfies $\delta_X=\delta H$ and hence $\delta\equiv\delta_X$.
This implies that $\delta$ is continuous with respect to the final topology of $\delta_X$
and hence $X=\delta(2^\IN)$ is compact.
\end{proof}

We note that the space $2^\IN$ itself is not multi-retraceable by Proposition~\ref{prop:multi-retraceability},
since not every computable function $f:\In\IN^\IN\to2^\IN$ has a total computable extension. 
Hence, compactness is far from being sufficient for multi-retraceability.

Since every space $\overline{X}$ is retraceable, it is also compact. 
We can say more in this case. The spaces $\overline{X}$ are also connected and hence
they can be seen as simultaneous one-point compactification and connectification of $X$.
However, topologically, this is 
not all too interesting, as the topology of $\overline{X}$ is always the indiscrete topology.
By $\widehat{n}=nnn...\in\IN^\IN$ we denote the constant sequence with value $n\in\IN$.

\begin{proposition}[Indiscrete topology]
Let $X$ be a represented space. The topology of $\overline{X}$ is $\{\emptyset,\overline{X}\}$.
\end{proposition}
\begin{proof}
On the one hand, the special point $\bot\in\overline{X}$ is a member of every non-empty open set, as any prefix of a name of a point
can be extended to a name of $\bot$ by extending it with zeros. 
More precisely, if $U\In\overline{X}$ is open and non-empty, then there exists a finite prefix $w\in\IN^*$ such that
$w\IN^\IN\In\delta_{\overline{X}}^{-1}(U)$ and hence $\bot=\delta_{\overline{X}}(w\widehat{0})\in U$.
On the other hand, there are names $p$ of $\bot$ that start just with zeros
and hence any finite prefix of $p$
can be extended to a name of any point in $X$. 
That is, if $\bot\in U$, then $\widehat{0}\in\delta_{\overline{X}}^{-1}(U)$ and hence
$0^np\in\delta_{\overline{X}}^{-1}(U)$ for every $p\in\IN^\IN$ and a suitable $n\in\IN$.
Since for every $x\in\overline{X}$ there is some $p$ with $\delta_{\overline{X}}(0^np)=x$, it follows 
that $x\in U$. 
Hence, the only possible non-empty open set is $\overline{X}$.
\end{proof}

For many represented spaces $X$ we cannot expect computable single-valued retractions $r:\overline{X}\to X$ to exist,
not even multi-valued ones.
Sometimes, there are at least retractions with weaker computability properties and we give two examples.

\begin{proposition}[Special retractions]
\label{prop:special-retractions}
Let $(X,\delta_X)$ be a represented space and $\delta_X$ total. 
There are retractions with the given properties:
\label{prop:retraction-N}
\begin{enumerate}
\item $r:\overline{\IN}\to\IN$ that is computable with finitely many mind changes,
\item $r:\overline{\IN^\IN}\to\IN^\IN$ that is limit computable, 
\item $r:\overline{X}\to X$ that is limit computable.
\end{enumerate}
\end{proposition}
\begin{proof}
(1) Given a name $p$ of a point in $\overline{\IN}$ we generate a name of $0\in\IN$ until we find
the first non-zero entry $n+1=p(i)$ for some $i\in\IN$, in which case we change our mind to a name of $n\in\IN$.
This describes a finite mind change computation of a retraction $r:\overline{\IN}\to\IN$.\\
(2) Given a name $p$ of a point in $\overline{\IN^\IN}$ we generate a name of $\widehat{0}\in\IN^\IN$ that we overwrite with
all digits of $p-1$ whenever we find non-zero content in $p$.
This describes a limit computation of a retraction $r:\overline{\IN^\IN}\to\IN^\IN$.\\
(3) Follows from (2) since $\delta_X$ is total. 
\end{proof}

\section{Total Weihrauch Reducibility}
\label{sec:total}

In this section we are going to recall the definition of Weihrauch reducibility and of total Weihrauch reducibility,
which was introduced in \cite{BG20}.
We write $F\vdash_\t f$, if $F$ is a total realizer of $f$.
We now recall the definition of ordinary and strong Weihrauch reducibility on problems $f,g$,
which is denoted by $f\leqW g$ and $f\leqSW g$, respectively, and we recall
the two new concepts of {\em total Weihrauch reducibility} and {\em strong total Weihrauch reducibility},
which are denoted by $f\leqTW g$ and $f\leqSTW g$, respectively.

\begin{definition}[Weihrauch reducibility]
\label{def:Weihrauch-reducibility}
Let $f:\In X\mto Y$ and $g:\In U\mto V$ be problems. We define:
\begin{enumerate}
\item $f\leqW g:\!\iff\!(\exists$ computable $H,K:\In\IN^\IN\to\IN^\IN)(\forall G\vdash g)\; H\langle \id,GK\rangle\vdash f$.
\item $f\leqSW g:\!\iff\!(\exists$ computable $H,K:\In\IN^\IN\to\IN^\IN)(\forall G\vdash g)\;HGK\vdash f$.
\item $f\leqTW g:\!\iff \!\!(\exists$ computable $H,K:\In\IN^\IN\to\IN^\IN)(\forall G\vdash_\t g)\; H\langle \id,GK\rangle\vdash_\t f$.
\item $f\leqSTW g:\!\iff (\exists$ computable $H,K:\In\IN^\IN\to\IN^\IN)(\forall G\vdash_\t g)\;HGK\vdash_\t f$.
\end{enumerate}
For (3) and (4) we assume that we replace each of the given representations of $X,Y,U$ and $V$ by a computably equivalent precomplete representation of the corresponding set. 
\end{definition}

We call the reducibilities $\leqW$ and $\leqSW$ {\em partial}
in order to distinguish them from their total counterparts $\leqTW$ and $\leqSTW$.
We note that precompleteness is not required or relevant in the partial case,
but it can be assumed without loss of generality since the concept of partial
(strong) Weihrauch reducibility is invariant under computably equivalent
representations~\cite[Lemma~2.11]{BG11}. 
In \cite[Corollary~4.3]{BG20} we have proved that due to precompleteness in the definition above
also $\leqTW$ and $\leqSTW$ are invariant under equivalent representations.
We have also proved that the definition does not depend on the choice of the precomplete
representation in the equivalence class and that it yields preorders $\leqTW$ and $\leqSTW$.

We have also proved in \cite[Corollary~4.7]{BG20} that the partial Weihrauch reductions imply their
total counterparts in the following sense.

\begin{corollary}[Partial and total Weihrauch reducibility]
\label{cor:partial-total}
Let $f$ and $g$ be problems. Then $f\leqW g\TO f\leqTW g$ and $f\leqSW g\TO f\leqSTW g$.
\end{corollary}

This means that all positive results that hold for a partial version of Weihrauch
reducibility can be transferred to the corresponding total variant.

We note that the reducibilities $\leqTW$ and $\leqSTW$ share similar
properties as $\leqW$ and $\leqSW$ when it comes to the preservation of 
computability or other properties.
We say that a class $\CC$ of problems is {\em preserved downwards} by a reducibility $\leq$ for problems
if $f\leq g$ and $g\in\CC$ imply $f\in\CC$.
In \cite[Proposition~4.9]{BG20} we proved that 
computability, continuity, limit computability, Borel measurability and non-uniform computability are preserved
downwards by $\leqTW$.

A class $\CC$ of functions $F:\In\IN^\IN\to\IN^\IN$ constitutes a property of problems that is preserved 
downwards by total Weihrauch reducibility if the following conditions are satisfied:
$\CC$ contains the identity, is closed under composition with computable functions, is closed under juxtaposition with
the identity and $\CC$ respects precompleteness. Later we prove that finite mind change computability
and Las Vegas computability are not preserved downwards by $\leqTW$, whereas non-deterministic computability is preserved downwards.

\section{Completion, Totalization and Co-Completion}
\label{sec:completion}

In this section we recall the definition of the closure operation $f\mapsto\overline{f}$ on (strong) 
Weihrauch reducibility that was introduced in \cite{BG20}
and we prove some further properties of it.
For the definition of the completion $\overline{f}$ we use the completion $\overline{X}$ of a represented space.

\begin{definition}[Completion]
\label{def:completion}
Let $f:\In X\mto Y$ be a problem. We define the {\em completion} of $f$ by
\[\overline{f}:\overline{X}\mto\overline{Y},x\mapsto\left\{\begin{array}{ll}
   f(x) & \mbox{if $x\in\dom(f)$}\\
   \overline{Y} & \mbox{otherwise}
\end{array}\right.\]
\end{definition}

We note that the completion $\overline{f}$ is always {\em pointed}, i.e., it has a computable
point in its domain. This is because $\bot\in\overline{X}$ is always computable (as it has the constant zero sequence as a name).

In \cite[Lemma~5.2]{BG20} we have proved that completion generates total Weihrauch reducibility in the following sense.

\begin{lemma}[Completion and total Weihrauch reducibility]
\label{lem:completion-total}
For all problems $f,g$:
$f\leqW\overline{g}\iff\overline{f}\leqW\overline{g}\iff f\leqTW g$ 
and $f\leqSW\overline{g}\iff\overline{f}\leqSW\overline{g}\iff f\leqSTW g$.
\end{lemma}

Thus, we could define total Weihrauch reducibility also using the completion operation and partial Weihrauch reducibility.

In \cite[Proposition~5.4]{BG20} we also proved that completion is a {\em closure operator} with respect to $\leqTW$,
i.e., $f\leqTW \overline{f}$, $\overline{\overline{f}}\leqTW \overline{f}$ and 
$f\leqTW g\TO \overline{f}\leqTW \overline{g}$. An analogous result holds for $\leqSTW$. 

It is clear that every $f$ is strongly totally equivalent to its completion.

\begin{corollary} 
\label{cor:completion-total}
$f\equivSTW\overline{f}$ for every problem $f$.
\end{corollary}

In the study of total Weihrauch reducibility the degrees that have identical cones
with respect to partial and total Weihrauch reducibility play an important r\^{o}le.
Hence, we introduce a name for such degrees.

\begin{definition}[Complete problems]
A problem $f$ is called {\em complete} if $f\equivW\overline{f}$ and
{\em strongly complete} if $f\equivSW\overline{f}$.
\end{definition}

It is straightforward to derive the following characterization of completeness,
which shows that from the perspective of a complete problem the lower cones in the 
Weihrauch lattice and the total Weihrauch lattice are indistinguishable (see also \cite[Theorem~5.7]{BG20}).

\begin{proposition}[Completeness]
\label{prop:completeness}
Let $g$ be a problem. Then the following hold:
\begin{enumerate}
\item $g$ complete $\iff (\forall$ problems $f)(f\leqW g\iff f\leqW \overline{g})$.
\item $g$ strongly complete $\iff (\forall$ problems $f)(f\leqSW g\iff f\leqSW \overline{g})$.
\end{enumerate}
\end{proposition}

Examples of complete problems are abundant. In \cite[Proposition~5.8]{BG20} completeness was proved, e.g., for
 the Turing jump operator $\J$ and and the binary sorting problem $\SORT$ that
was introduced and studied by Neumann and Pauly~\cite{NP18}.
It was also proved that problems such as $\WBWT_2,\ACC_X,\PA$ and $\MLR$ (see \cite{BHK17a} for definitions)
are complete.
We will see many further examples in form of complete choice problems that we study systematically
in section~\ref{sec:choice}.

\begin{proposition}[Complete problems]
\label{prop:complete-problems}
The following problems are all strongly complete:
\begin{enumerate}
\item $\J:\IN^\IN\to\IN^\IN,p\mapsto p'$,
\item $\lim:\In\IN^\IN\to\IN^\IN,\langle p_0,p_1,p_2,...\rangle\mapsto\lim_{n\to\infty} p_n$,
\item $\LPO:\IN^\IN\to\{0,1\},\LPO(p)=0:\iff(\exists n\in\IN)\;p(n)=0$,
\item $\SORT:2^\IN\to2^\IN$ with
   \[\SORT(p):=\left\{\begin{array}{ll}
    0^k\widehat{1} & \mbox{if $p$ contains exactly $k\in\IN$ zeros}\\
    \widehat{0}       & \mbox{if $p$ contains infinitely many zeros}
\end{array}\right..\]
\item $\WBWT_2:2^\IN\mto2^\IN,p\mapsto\{q\in2^\IN:\lim_{n\to\infty}q(n)$ is a cluster point of $p\}$.
\end{enumerate}
\end{proposition}

These results show that the cones below the given problems are identical in the total and partial Weihrauch lattices.
It is known, for instance, that $f$ is limit computable if and only if $f\leqW\lim$ \cite{BGP18}. 
Hence, an analogous statement holds for $\leqTW$, since $\lim$ is complete.

One should not come to the incorrect conclusion that all functions are complete. 
Here is a counter example, which is based on the fact that $\J^{-1}:\In\IN^\IN\to\IN^\IN$ has no computable points in its domain.

\begin{example}
$\J^{-1}\lW\overline{\J^{-1}}$.
\end{example}

The operation of completion is somewhat related to totalization\footnote{The totalization was studied by Brattka, Le Roux and Pauly (unpublished work 2012).}.
Totalization is not a closure operator, but it is sometimes easier to handle because it does not involve completions of spaces.

\begin{definition}[Totalization]
For every problem $f:\In X\mto Y$ we denote by 
\[\T f:X\mto Y,x\mapsto\left\{\begin{array}{ll}
   f(x) & \mbox{if $x\in\dom(f)$}\\
   Y & \mbox{otherwise}
\end{array}\right.\]
the {\em total version} or {\em totalization} of $f$. 
\end{definition}

If $(X,\delta_X)$ and $(Y,\delta_Y)$ are represented spaces and $f:\In X\mto Y$ is a problem,
then we call $f^\r:=\delta_Y^{-1}\circ f\circ\delta_X:\In\IN^\IN\mto\IN^\IN$ the
{\em realizer version} of $f$. It satisfies $f^\r\equivSW f$ since $f^\r$ has exactly
the same realizers as $f$. So, $f^\r$ can be seen as the Baire space version of $f$.
It is clear that totalization is closely related to the completion, as we have the following obvious result.

\begin{lemma}[Completion and totalization]
\label{lem:completion-totalization-realizer}
$\overline{f}\equivSW \T f^\r$ holds for every problem $f:\In X\mto Y$ provided $f^\r$ is formed with respect
to the precompletions of the original representations of $X$ and $Y$.
\end{lemma}
\begin{proof}
We consider the represented spaces $(X,\delta_X)$ and $(Y,\delta_Y)$.
We obtain the realizer version ${\overline{f}\,}^\r:\IN^\IN\mto\IN^\IN$, given by
\[{\overline{f}\,}^\r(p)=\delta_{\overline{Y}}^{-1}\circ\overline{f}\circ\delta_{\overline{X}}(p)
=\left\{\begin{array}{ll}
  (\delta_Y^\wp)^{-1}\circ f\circ\delta_X^\wp(p) & \mbox{if $p\in\dom(f\circ\delta_X^\wp)$}\\
  \IN^\IN & \mbox{otherwise}
\end{array}\right..\]
If  $f^\r=(\delta_Y^\wp)^{-1}\circ f\circ\delta_X^\wp$, then we obtain
$\overline{f}\equivSW{\overline{f}\,}^\r=\T f^\r$.
\end{proof}

In other words, completion can be seen as a totalization of the realizer version with respect to precomplete representations.
More generally, the two operations of completion and totalization coincide under certain relatively special assumptions.

\begin{lemma}[Completion and totalization]
\label{lem:completion-totalization}
Let $f:\In X\mto Y$ be a problem. Then:
\begin{enumerate}
\item $\overline{f}\leqSW\T f$ if $X$ is multi-retraceable,
\item $\T f\leqSW\overline{f}$ if $Y$ is multi-retraceable.
\end{enumerate}
\end{lemma}
\begin{proof}
We note that $\iota_X:X\to\overline{X},x\mapsto x$ and $\iota_Y:Y\to\overline{Y},y\mapsto y$ are computable by Corollary~\ref{cor:completion}.
If $X$ is multi-retraceable, then by definition there is a computable retraction
$r_X:\overline{X}\mto X$.
It is clear that $\iota_Y\circ\T f\circ r_X(x)\In\overline{f}(x)$ for all $x\in\overline{X}$ and hence $\overline{f}\leqSW\T f$.
If $Y$ is multi-retraceable, then there is a computable retraction
$r_Y:\overline{Y}\mto Y$.
In this case we obtain $r_Y\circ\overline{f}\circ\iota_X(x)=\T f(x)$ for all $x\in X$ and hence $\T f\leqSW\overline{f}$.
\end{proof}

The second condition can also be converted into a characterization of multi-retraceability.
This follows if one applies it to the partial computable function $f:=\iota_{Y}^{-1}:\In \overline{Y}\to Y$,
since $\T f:\overline{Y}\mto Y$ is a retraction.

\begin{corollary}[Multi-retraceability]
A represented space $Y$ is multi-retraceable if and only if $\T f\leqW\overline{f}$ holds
for all problems $f:\In X\mto Y$.
\end{corollary}

If a problem $g$ is already total, then $\T g=g$ and hence we obtain the following
characterization of completeness. The given conditions are necessary since $g=\overline{f}$
satisfies them (because the completion $\overline{X}$ of any space $X$ is retraceable by Corollary~\ref{cor:completion}).

\begin{corollary}[Completeness]
\label{cor:completeness}
A problem $f$ is complete if and only if 
$f\equivW g$ for some total problem $g:X\mto Y$ 
on multi-retraceable spaces $X$ and $Y$.
\end{corollary}

An analogous result holds for strong completeness and $\equivSW$.
The characterization of completeness given in Proposition~\ref{prop:completeness} suggests
a dual notion of co-completeness that we define together with the related notion of co-totality.
From the perspective of a co-complete problem the upper cones in the Weihrauch lattice and
the total Weihrauch lattice are indistinguishable.

\begin{definition}[Co-completeness and co-totality]
\label{def:co-complete-total}
Let $f$ be a problem.
\begin{enumerate}
\item $f$ is called {\em co-complete} if $f\leqW\overline{g}\iff f\leqW g$ holds for all problems $g$.
\item $f$ is called {\em co-total} if $f\leqW\T g\iff f\leqW g$ holds for all problems $g$.
\end{enumerate}
Likewise we define {\em strongly co-complete} and {\em strongly co-total} problems $f$ with the
help of $\leqSW$ instead of $\leqW$.
\end{definition}

Due to Lemma~\ref{lem:completion-totalization-realizer} we obtain that co-totality implies co-completeness.

\begin{corollary}[Co-completeness and co-totality]
\label{cor:co-completeness-totality}
Every (strongly) co-total problem $f:\In X\mto Y$ is (strongly) co-complete.
If $Y$ is multi-retraceable, then the inverse implication holds true as well.
\end{corollary}

In Corollaries~\ref{cor:LPOS-co-total} and \ref{cor:WFTS-co-total} we will see examples that witness that the inverse implication does not hold
in general.
There is an obvious relation between co-completeness of the completion and completeness.

\begin{lemma}[Completeness and co-completeness]
\label{lem:complete-co-complete}
$\overline{f}$ co-complete $\TO f$ complete and $\overline{f}$ strongly co-complete $\TO f$ strongly complete  hold for every problem $f$.
\end{lemma}

There is a useful condition that implies strong co-completeness.
We call a problem $f:\In X\mto Y$ {\em diverse} if for all $x\in\dom(f)$ there exists a $y\in\dom(f)$
such that $f(x)\cap f(y)=\emptyset$.

\begin{proposition}[Diversity]
\label{prop:diversity}
Every diverse problem $f:\In X\mto Y$ is strongly co-complete, and if the representation of $Y$ is total,
then $f$ is also strongly co-total.
\end{proposition}
\begin{proof}
Let $g:\In W\mto Z$ be an arbitrary problem.
It is clear that $f\leqSW g$ implies $f\leqSW\overline{g}$ since $g\leqSW\overline{g}$.
We now consider computable witnesses $H,K:\In\IN^\IN\to\IN^\IN$ for $f\leqSW\overline{g}$.
Now, let us suppose that $G\vdash g$ holds with respect to the precompletions of the underlying representations of $W$ and $Z$, and let us assume that $\dom(G)$ contains exactly all names of points in $\dom(g)$.
Every total extension $G'$ of $G$ satisfies $G'\vdash\overline{g}$ and hence $HG'K\vdash f$.
Let $G'$ be such a total extension. We claim that also $HGK\vdash f$. 
Let us assume that this is not the case. Then there is some name $p$ of a point $x\in\dom(f)$ such that $HGK(p)$ is not a name of a point in $f(x)$.
If $K(p)\in\dom(G)$, then $G'K(p)=GK(p)$.
Since $K(p)$ is a name of a point in $\dom(g)$, this implies that $HGK(p)=HG'K(p)$ is a name of a point in $f(x)$, which is a contradiction. 
This implies that $K(p)\not\in\dom(G)$. But then there is a total extension
$G''$ of $G$ such that $HG''K(p)$ is a name of some point in $f(y)$ for some $y\in\dom(f)$ such that $f(x)\cap f(y)=\emptyset$,
since $f$ is diverse.
This is a contradiction to $HG''K\vdash f$.
Hence, the assumption was wrong and we actually have $HGK\vdash f$.
This proves $f\leqSW g$ and altogether $f$ is strongly co-complete.
Almost the same proof shows that $f$ is also strongly co-total, provided that the representation of $Y$ is total,
since this ensures that the extensions $G'$ and $G''$ select valid names. 
\end{proof}

Since single-valued problems that are not constant are diverse, we obtain the following corollary.

\begin{corollary}[Single-valuedness]
\label{cor:single-valuedness}
Let $f:\In X\to Y$ be a single-valued problem that is not constant, then $f$ is strongly co-complete,
and if the representation of $Y$ is total, then $f$ is also strongly co-total.
\end{corollary}

We note that constant computable problems $f$ are not (strongly) co-complete (except for the nowhere defined problems),
because if $p$ is a name of a point in the domain of $f$, then there is a problem $g$ that has no point in the 
domain that can be computed from $p$ and hence $f\nleqW g$, while $f\leqW\overline{g}$.
Likewise, $\id$ is not co-complete and hence diversity is not sufficient for co-completeness in the non-strong case.

We close this section with some remarks on the completion of jumps.
As a preparation we study composition.
It is easy to see that the completion of a composition is equal to 
the composition of the completions. 

\begin{lemma}[Composition]
\label{lem:composition}
Let $f:\In X\mto Y$ and $g:\In Y\mto Z$ be problems.
Then $\overline{g}\circ\overline{f}=\overline{g\circ f}$.
\end{lemma}

We recall that the {\em jump} $f':\In X\mto Y$ of a problem $f:\In X\mto Y$ is defined
exactly as $f$, but the represented space $(X,\delta_X)$ on the input side is replaced
by $X'=(X,\delta_{X}')$, where $\delta_X':=\delta_X\circ\lim$.
The jump was defined in \cite{BGM12} and it is easy to see that for problems
of type $f:\In\IN^\IN\mto\IN^\IN$ we have $f'\equivSW f\circ\lim$ (see also \cite[Lemma~5.2]{BGM12} for the
single-valued case).
We can now draw some conclusions on the completion of a jump.
In particular, the jumps of strongly complete problems are strongly complete, which yields
many further examples of complete problems.

\begin{proposition}[Jumps]
\label{prop:jumps}
$\overline{f'}\leqSW{\overline{f}\,}'\equivSW\overline{{\overline{f}\,'}}$ holds for all problems $f$.
In particular, the jump $f'$ of every strongly complete problem $f$ is strongly complete again.
\end{proposition}
\begin{proof}
We first prove $\overline{f'}\leqSW{\overline{f}\,}'$ for problems of type $f:\In\IN^\IN\mto\IN^\IN$.
For such problems we have $f'\equivSW f\circ\lim$ and hence 
$\overline{f'}\equivSW\overline{f\circ\lim}=\overline{f}\circ\overline{\lim}$ by Lemma~\ref{lem:composition}.
By Proposition~\ref{prop:complete-problems} $\overline{\lim}\equivSW\lim$ and hence $\overline{\lim}$
is limit computable, i.e., it has a realizer of the form $\lim\circ K$ with a computable $K$.
This means that $\delta_{\overline{\IN^\IN}}\circ\lim\circ K(p)\in\overline{\lim}\circ\delta_{\overline{\IN^\IN}}(p)$
for all $p\in\IN^\IN$.
We also consider $\overline{f}^\r=\delta_{\overline{\IN^\IN}}^{-1}\circ\overline{f}\circ\delta_{\overline{\IN^\IN}}$ and we obtain
${\overline{f}\,}^\r\circ\lim\circ K(p)=\delta_{\overline{\IN^\IN}}^{-1}\circ\overline{f}\circ\delta_{\overline{\IN^\IN}}\circ\lim\circ K(p)\in\delta_{\overline{\IN^\IN}}^{-1}\circ\overline{f}\circ\overline{\lim}\circ\delta_{\overline{\IN^\IN}}(p)=(\overline{f\circ\lim})^\r(p)$, which implies
$\overline{f'}\equivSW{\overline{f\circ\lim}\,}^\r\leqSW{\overline{f}\,}^\r\circ\lim\equivSW{{\overline{f}\,}^\r}'\equivSW{\overline{f}\,}'$.
For a general problem $f:\In X\mto Y$ we also obtain $\overline{f'}\leqSW{\overline{f}\,}'$, since we can apply
the result above to $f^\r:\In\IN^\IN\mto\IN^\IN$, where we use that jumps and completions are monotone with respect to $\leqSW$.
This reduction also implies $\overline{{\overline{f}\,}'}\leqSW{\overline{\overline{f}}\,}'\leqSW{\overline{f}\,}'$, since completion
is a closure operator. The inverse reduction holds for the same reason, i.e., ${\overline{f}\,}'\equivSW\overline{{\overline{f}\,'}}$.
This shows that the jump $f'$ of a strongly complete problem $f\equivSW\overline{f}$ is strongly complete.
\end{proof}

This result requires strong completeness, since jumps are not monotone with respect to $\leqW$ in general.
We also note that the jump of a total problem is an upper bound of its completion, provided that the
input space has a total representation. This follows from Proposition~\ref{prop:special-retractions}~(3).

\begin{corollary}
\label{cor:completion-jump}
$\overline{f}\leqSW f'$ for every total problem $f:X\mto Y$ such that $X$ has a total representation.
\end{corollary}

It follows from Theorem~\ref{thm:choice-Baire} that this result cannot be generalized to partial problems. 
With the help of jumps we can also express a sufficient criterion that guarantees co-completeness and co-totality
for certain single-valued maps.
We use Sierpi\'nski space $\IS=\{0,1\}$ that is represented by $\delta_\IS(p)=0:\iff p=\widehat{0}$.
We mention that $\delta_\IS$ is an example of a precomplete total representation.

\begin{proposition}[Single-valuedness]
\label{prop:single-valuedness}
Let $f:\In X\to\IN$, $f_\IS:\In X\to\IS$ be single-valued problems that are not constant. Then:
\begin{enumerate}
\item $f'$ is co-complete and co-total.
\item $f_\IS'$ is co-complete.
\end{enumerate}
\end{proposition}
\begin{proof}
(1) It suffices to prove that $f'$ is co-total as this implies co-completeness by Corollary~\ref{cor:co-completeness-totality}.
Hence, let $f'\leqW\T g$ hold for some problem $g:\In W\mto Z$ via computable $H,K:\In\IN^\IN\to\IN^\IN$.
We consider a name $p$ of a point $x\in\dom(f')$, and we let $n:=f'(x)$.
Then there is name $r$ of a point in $Z$ such that $H\langle p,r\rangle$ is a name of $n$.
Hence by continuity of $H$, a finite prefix $w\prefix p$ is sufficient to guarantee that any extension of $w$ is mapped
with $r$ by $H$ to $n$. Since $f'$ is not constant, there is a $y\in\dom(f')$ with $k:=f'(y)\not=n$
and since we use the jump of a representation for $X'$, we have that there is a name
$q\in w\IN^\IN$ of $y$. Suppose $q$ is mapped by $K$ to a point outside of $\dom(g)$.
Then there is a realizer $G\vdash\T g$ with $GK(q)=r$ and hence $H\langle q,GK(q)\rangle=n$, which is incorrect.
Hence, all names $q\in w\IN^\IN$ of points $y\in\dom(f')$ with $f'(y)\not=n$ are mapped
by $K$ to points inside of $\dom(g)$. With a similar argument as before there is also
some name $s$ of a point in $Z$ such that $H\langle q,s\rangle$ is a name of $k$
and by continuity of $H$ a finite prefix $v$ of $q$ is sufficient to guarantee that $H$ maps
any extension of $v$ with $s$ to a name of $k$. We can assume $w\prefix v$. 
As before there cannot be any
name $t\in v\IN^\IN$ of a point $z\in\dom(f')$ with $f'(z)\not=k$ that is mapped by $K$ to a point
outside of $\dom(g)$. Altogether, there is no name $t\in v\IN^\IN$ of some point $z\in\dom(f')$
whatsoever that is mapped by $K$ to a point outside of $\dom(g)$.
Since there is a computable
function $F:\In\IN^\IN\to\IN^\IN$ that maps every name $t$ of a point inside of $\dom(f')$ to a name $F(t)\in v\IN^\IN$
of the same point, we obtain that $H\langle F,GKF\rangle\vdash f'$ for every $G\vdash g$, i.e., $f'\leqW g$.
This proves that $f'$ is co-total.\\
(2) Let $f_\IS'\leqW\overline{g}$ for some problem $g:\In W\mto Z$ hold via computable functions $H,K:\In\IN^\IN\to\IN^\IN$.
We consider a name $p$ of a point $x\in\dom(f_\IS')$ with $f_\IS'(x)=1$. Such a point must exist since $f_\IS'$ is not constant.
Let us assume that $K(p)$ is a name of a point outside of $\dom(g)$. 
Then there is a realizer $G\vdash\overline{g}$ such that $GK(p)=\widehat{0}$, which is a name of $\bot\in\overline{Z}$.
Now $H\langle p,GK(p)\rangle$ is a name of $1\in\IS$ and by continuity of $H$ finite prefixes $w\prefix p$ and $v\prefix \widehat{0}$
suffice to ensure that $H$ maps the corresponding extensions to $1\in\IS$.
We note that  $v\prefix\widehat{0}$ can be extended to a name of any point in $\overline{Z}$, given the way $\overline{Z}$ is represented, 
and $w\prefix p$ can also be extended to a name of any point in $X'$, as we use the jump of a representation.
Suppose now that $q\in w\IN^\IN$ is a name of a point $y\in\dom(f_\IS')$ such that $f_\IS'(y)=0\in\IS$.
Such a point exists since $f_\IS'$ is not constant.
Then there is a realizer $G_1\vdash\overline{g}$ such that $v\prefix G_1K(q)$ and hence
$H\langle q,G_1K(q)\rangle$ is a name of $1\in\IS$, which is incorrect.
Hence all names $p$ of points $x\in\dom(f_\IS')$ with $f_\IS'(x)=1\in\IS$
are mapped by $K$ to names $K(p)$ of points inside of $\dom(g)$.
In particular there exists a name $p$ of a point $x\in\dom(f'_\IS)$ with $f_\IS'(x)=1$
and a corresponding name $r$ of a point in $\overline{Z}$ such that $H\langle p,r\rangle=1$
and finite prefixes $w\prefix p$ and $v\prefix r$ suffice to ensure that $H$ maps
all extensions to $1\in\IS$. With a similar argument as above one can show
that no name $q\in w\IN^\IN$ of a point $y\in\dom(f_\IS')$ with $f_\IS'(y)=0$ can be mapped
by $K$ to a name $K(q)$ of a point outside of $\dom(g)$.
Otherwise, there would be a realizer $G_2\vdash\overline{g}$ with
$G_2K(q)=r$ and hence $H\langle q,G_2K(q)\rangle$ is a name of $1$, which is incorrect.
Hence altogether, each name $q\in w\IN^\IN$ of a point in $\dom(f_\IS')$ is mapped by $K$ to a name $K(q)$ of a point inside of $\dom(g)$.
Using a computable function $F$ that maps any name $p$ of a point in $\dom(f'_\IS)$ to
a name $q\in w\IN^\IN$ of the same point,
we obtain that $H\langle F,GKF\rangle\vdash f_\IS'$ for every $G\vdash g$, i.e., $f_\IS'\leqW g$.
This proves that $f_\IS'$ is co-complete.
\end{proof}

Using this result it is now easy to provide some interesting examples of co-complete and co-total problems.
We note that $\lim_\IN\equivSW\id_\IN'$.

\begin{corollary}[Co-complete and co-total problems]
\label{cor:co-complete-co-total}\
\begin{enumerate}
\item $\lim_\IN:\In\IN^\IN\to\IN,p\mapsto\lim_{n\to\infty}p(n)$ is co-total and co-complete,
\item $\LPO':(\IN^\IN)'\to\{0,1\}$ is co-total and co-complete,
\item $\LPO_\IS':(\IN^\IN)'\to\IS,p\mapsto\LPO(p)$ is co-complete.
\end{enumerate}
\end{corollary}

In Corollary~\ref{cor:LPOS-co-total} we will see that $\LPO_\IS'$ is not co-total.
Hence co-totality and co-completeness are actually not equivalent conditions
and Proposition~\ref{prop:single-valuedness}~(2) cannot be strengthened to co-totality.

There is a useful characterization of $\LPO'$ as the infinity problem. By 
\[\INF:\IN^\IN\to\{0,1\},p\mapsto\left\{\begin{array}{ll}
1 & \mbox{if $p(n)=0$ for infinitely many $n\in\IN$}\\
0 & \mbox{otherwise}
\end{array}\right.\]
we denote the {\em infinity problem}. By $\INF_\IS:\IN^\IN\to\IS$ we denote the analogous problem with output space $\IS$.
These two problems were already studied under the names ${\text{\rm\sffamily isInfinite}}$ and ${\text{\rm\sffamily isInfinite}_\IS}$
by Neumann and Pauly~\cite{NP18}.
The following is easy to see.

\begin{lemma}[Infinity problem]
\label{lem:infinity}
$\LPO'\equivSW\INF$ and $\LPO'_\IS\equivSW\INF_\IS$.
\end{lemma}
\begin{proof}
Given $p\in\IN^\IN$, we let $K(p)\langle n,k\rangle=0$ if the word $p(0)...p(n)$ contains the digit $0$ less than $k$ times
and $K(p)\langle n,k\rangle=1$ otherwise. Then 
\[(\forall k\in\IN)\;(\lim K(p))(k)=\lim_{n\to\infty}K(p)\langle n,k\rangle\not=0\] 
if and only if $0$ appears infinitely often in $p$. Hence $\LPO'\circ K(p)=\INF(p)$.
Since $K$ is computable, this proves $\INF\leqSW\LPO'$ and $\INF_\IS\leqSW\LPO'_\IS$.
Vice versa, given $p:=\langle p_0,p_1,p_2,...\rangle\in\dom(\lim)$ we can enumerate the numbers $1,2,3,...,$ into $K(p)$, 
and for each $\langle n,k\rangle=0,1,2,...$ after the other we do the following: whenever we find an $i\geq k$ such that $p_i(n)\not=0$, then
we enumerate $0$ into $K(p)$, and only in this case we move to the next $\langle n,k\rangle$. 
Hence $K(p)$ contains infinitely many zeros 
if and only if $(\forall n,k\in\IN)(\exists i\geq k)\;p_i(n)\not=0$,
which holds if and only if $(\forall n\in\IN)(\lim_{i\to\infty} p_i)(n)\not=0$, i.e., $\LPO'(p)=\INF\circ K(p)$.
This proves $\LPO'\leqSW\INF$ and $\LPO_\IS'\leqSW\INF_\IS$.
\end{proof}

In particular, we can conclude that $\INF$ is co-total and $\INF_\IS$ is co-complete.
$\INF$ is clearly not limit computable, since 
\[\INF^{-1}\{1\}=\{p\in\IN^\IN:(\forall k\in\IN)(\exists n\geq k)\;p(n)=0\}\] 
is known to be $\Pi^0_2$--complete (see, e.g., \cite[Exercise~23.1]{Kec95}).

\begin{lemma}
\label{lem:LPO-lim}
$\LPO'\nleqW\lim$.
\end{lemma}

It is also easy to see that there is a retraction $r:\overline{\IN^\IN}\to\IN^\IN$
that is computable with the help of $\INF$. This is because $\INF$ can detect
whether a name $p$ of a point in $\overline{\IN^\IN}$ is actually a name of a point in $\IN^\IN$, i.e., whether $p(n)\not=0$ for infinitely many $n$.

\begin{lemma}
\label{lem:retraction-LPO}
There is a retraction $r:\overline{\IN^\IN}\to\IN^\IN$ with $r\leqW\LPO'$.
\end{lemma}

\section{Choice Problems}
\label{sec:choice}

Choice principles form the backbone of the Weihrauch lattice, and many other problems
can be classified by proving their equivalence to a suitable choice problem~\cite{BGP18}.
Hence, it is important to understand which choice principles are complete
in order to see how the picture for the total Weihrauch lattice changes
compared to the partial version.

In order to recall the definition of choice we need to introduce the set $\AA(X)$ of closed subsets
of a topological space $X$. Typically, we will consider computable metric spaces $(X,d,\alpha)$ that are represented by their Cauchy representation~\cite{Wei00,BGP18}.
We denote by $B(x,r):=\{y\in X:d(x,y)<r\}$ the open ball with center $x\in X$ and radius $r\geq0$. 
More specifically, we denote by $B_{\langle n,\langle i,k\rangle\rangle}:=B(\alpha(n),\frac{i}{k+1})$
a basic open ball. A representation $\delta_{\AA_-(X)}$ of the set $\AA(X)$ can now be defined by
$\delta_{\AA_-(X)}(p):=X\setminus \bigcup_{n=0}^\infty B_{p(n)}$.
We denote the represented space $(\AA(X),\delta_{\AA_-(X)})$ for short by $\AA_-(X)$, where the ``$-$'' refers to negative information.
The computable points in $\AA_-(X)$ are known as {\em co-c.e.\ closed sets} and also as {\em $\Pi^0_1$--classes} in the case of $X=\IN^\IN$.
Since there are numbers $n\in\IN$ with $B_n=\emptyset$, it is easy to see that the representation $\delta_{\AA_-(X)}$ is precomplete and it is also total.

\begin{lemma}
\label{lem:precomplete-closed}
Let $X$ be a computable metric space. Then $\delta_{\AA_-(X)}$ is a precomplete and total representation of $\AA_-(X)$.
In particular, $\AA_-(X)$ is multi-retraceable.
\end{lemma}

The choice problem $\C_X$ of a given space $X$ is the problem of finding a point in a given closed $A\In X$.
By choosing appropriate spaces $X$ one obtains several important Weihrauch degrees. 
There are ways of extending the definition of $\delta_{\AA_-(X)}$ to other represented spaces
than computable metric ones~\cite{BBP12}. We are not going to use these extensions here, hence
the following definition is typically used for computable metric spaces $X$.

\begin{definition}[Choice]
\label{def:choice}
The problem
$\C_X:\In\AA_-(X)\mto X,A\mapsto A$, defined on $\dom(\C_X):=\{A:A\not=\emptyset\}$ is called the {\em choice problem} of the represented space $X$.
\end{definition}

Here the description $A\mapsto A$ of the map is to be read such that on the input side $A\in\AA_-(X)$ is a point of the input space, whereas on the output side it is a subset $A\In X$ of possible results.
Many restrictions of the choice problem have been considered. For instance, $\ConC_X$ denotes {\em connected choice}, i.e., $\C_X$ restricted to non-empty connected closed subsets $A\In X$.

One of our goals is to understand the completions $\overline{\C_X}$ of choice problems. 
Fortunately, the conditions given in Lemma~\ref{lem:completion-totalization}~(1)
are satisfied by Lemma~\ref{lem:precomplete-closed} for all choice principles of computable metric spaces.
We even obtain the following.

\begin{corollary}[Completion of choice]
\label{cor:completion-choice}
$\C_X\leqSW\overline{\C_X}\leqSW\T\C_X\equivSW\overline{\T\C_X}$ for all computable metric spaces $X$.
\end{corollary}
\begin{proof}
By Corollary~\ref{cor:completion} there is a computable embedding $\iota:X\to\overline{X}$. Hence, we obtain
$f:=\iota\circ\T\C_X\equivSW\T\C_X$. Moreover, $f:\AA_-(X)\mto\overline{X}$ has a total precomplete
representation on the input side by Lemma~\ref{lem:precomplete-closed} and a total precomplete
representation on the output side by Corollary~\ref{cor:completion}.
Hence, $\T\C_X$ is strongly complete by Corollary~\ref{cor:completeness} and Proposition~\ref{prop:multi-retraceability}.
The reduction $\C_X\leqSW\T\C_X$ is obvious, the other reductions follow since completion is a closure operator.
\end{proof}

An analogous statement holds true if $\C_X$ is replaced by any restriction such as $\ConC_X$ in all occurrences.
The advantage of this result is that $\T\C_X$ is conceptually
simpler than $\overline{\C_X}$, as it does only involve the original spaces and no completions.
We can also characterize the completion of the jump of choice.

\begin{proposition}[Completion of jumps of choice]
\label{prop:completion-jump-choice}
$\overline{\C_X'}\equivSW{\overline{\C_X}\,}'$ holds for every computable metric space $X$.
\end{proposition}
\begin{proof}
We have $\overline{\C_X'}\leqSW{\overline{\C_X}\,}'$ by Proposition~\ref{prop:jumps}.
We need to prove the inverse reduction. Given a name $p=\langle p_0,p_1,p_2,...\rangle$ such that
$\lim(p)$ exists, we can compute $K(p)$ that replaces all numbers $0$
in $p$ by a fixed number $k+1$ such that $B_k=\emptyset$.
Then also $\lim(K(p)-1)$ exists and if $\lim(p)-1\in\dom(\delta_{\AA_-(X)})$, then
\[\delta_{\AA_-(X)}(\lim(p)-1)=\delta_{\AA_-(X)}(\lim(K(p)-1)).\] 
In the case that $\lim(p)-1$ is only a finite word, $\lim(K(p)-1)$ is a name of some set.
Since a realizer of ${\overline{\C_X}\,}'$ with such an input $p$ can produce any name of a point in $\overline{X}$ as an output,
the reduction also works in this case.
Altogether, this proves ${\overline{\C_X}\,}'\leqSW\overline{\C_X'}$.
\end{proof}

Again, an analogous statement holds if $\C_X$ is replaced by any restriction of it.

\section{Choice on Compact Spaces}
\label{sec:compact}

Even though the assumptions of Lemma~\ref{lem:completion-totalization}~(2) are not satisfied in many cases,
we can often even prove $\T\C_X\leqSW\C_X$ using a computable multi-valued retraction $r:\AA_-(X)\mto\dom(\C_X)$.
We illustrate this with choice on Cantor space $2^\IN$.

\begin{proposition}[Choice on Cantor space]
\label{prop:compact-choice}
$\C_{2^\IN}\equivSW\overline{\C_{2^\IN}}\equivSW\T\C_{2^\IN}$.
\end{proposition}
\begin{proof}
We consider $X=2^\IN$.
By Corollary~\ref{cor:completion-choice} it suffices to prove $\T\C_{X}\leqSW\C_{X}$.
Firstly, we note that the set
\[B:=\{\langle k,\langle n_0,...,n_k\rangle\rangle\in\IN:B_{n_0}\cup...\cup B_{n_k}=X\}\]
is computable, as we can easily check whether there is a point $x\in X$ that is not covered by $B_{n_0}\cup...\cup B_{n_k}$.
Hence, given a list $p\in\IN^\IN$ of balls $B_{p(i)}$ with $A=X\setminus\bigcup_{i\in\IN}B_{p(i)}$ we can check
for every $i\in\IN$ whether $B_{p(0)}\cup...\cup B_{p(i-1)}\not=X$ and $B_{p(0)}\cup...\cup B_{p(i)}=X$.
Since $X$ is compact, this test will eventually be positive if and only if $A=\emptyset$.
As soon as this happens, we modify $p$ to $q$ such that $q(j):=p(j)$ for $j<i$ and $q(j):=p(i-1)$ for $j\geq i-1$  (where we assume that $B_{p(-1)}=\emptyset$ if $i=0$).
The map $p\mapsto q$ is a computable realizer for a multi-valued retraction $r:\AA_-(X)\mto\dom(\C_X)$ onto the non-empty sets,
i.e., $r(A)=A$ for $A\not=\emptyset$ and $r(\emptyset)\not=\emptyset$.
Such a retraction is all what is needed to prove $\T\C_X\leqSW\C_X$.
\end{proof}

The problems $f\leqW\C_{2^\IN}$ have been characterized in \cite{BBP12} exactly as the non-deter\-min\-istically
computable problems. Hence we obtain the following by Lemma~\ref{lem:completion-total}.

\begin{corollary}[Non-deterministic computability]
\label{cor:non-deterministic}
Non-deterministic computability is preserved downwards by total Weihrauch reducibility.
\end{corollary}

This proof of Proposition~\ref{prop:compact-choice} can be seen as a prototype of a completeness proof for a choice principle and several other choice
principles can be proved to be complete in a similar manner. There are three essential points that one needs
to check: whether the space $X$ is compact, whether a suitable set $B\In\IN$ is computable and whether
there is a suitable computable retraction $r$.
This is the case, for instance, for all non-empty finite spaces $X=n=\{0,...,n-1\}$ with $n\in\IN$ with essentially the same
proof as above.

\begin{proposition}[Finite choice]
\label{prop:finite-choice}
$\C_n\equivSW\overline{\C_n}\equivSW\T\C_n$ for all $n\geq1$.
\end{proposition}

We note that the above proof does not work in case of $n=0=\emptyset$, since then there is no possible retraction $r$.
In this case we have $\overline{\C_0}\equivSW\C_1$, which one can easily check directly.
We can conclude from this result that the principles $\C_n$ form a strictly increasing chain with respect to total Weihrauch reducibility (as they do with respect to partial Weihrauch reducibility by~\cite[Theorem~5.4]{Wei92c}).

\begin{corollary} $\C_n\lTW\C_{n+1}$ for all $n\geq1$.
\end{corollary}

In the following we use the {\em parallelization} $\widehat{f}$ of a problem $f$ and the {\em finite parallelization} $f^*$
in a purely non-technical way. Hence we refer the reader to \cite{BGP18} for the definitions.
The choice principle $\C_2$ is also known as $\LLPO$ and hence Proposition~\ref{prop:finite-choice} shows that 
$\LLPO$ is strongly complete. We could also use the fact that the parallelization of a complete problem is complete
(by \cite[Proposition~6.3]{BG20}) to derive Proposition~\ref{prop:compact-choice},
as is known that $\C_{2^\IN}\equivSW\widehat{\C_2}$~\cite[Theorem~8.5]{BG11}.
Likewise it is known that the compact choice principle $\K_\IN$ can be characterized by $\K_\IN\equivSW\C_2^*$~\cite[Proposition~10.9]{BGM12}, which we take
as the definition of $\K_\IN$ for the purposes of this article, and hence, with the help of the fact that
finite parallelization and jumps preserves (strong) completeness (by \cite[Proposition~6.3]{BG20} and Proposition~\ref{prop:jumps}), we arrive at the following conclusion.

\begin{corollary} 
\label{cor:KN}
$\K_\IN\equivSW\overline{\K_\IN}$ and $\K_\IN'\equivSW\overline{\K_\IN'}$.
\end{corollary}

The proof of Proposition~\ref{prop:compact-choice} can also be transferred to the case of connected choice.

\begin{proposition}[Connected choice]
\label{prop:connected-choice}
$\ConC_{[0,1]}\equivSW\overline{\ConC_{[0,1]}}\equivSW\T\ConC_{[0,1]}$.
\end{proposition}
\begin{proof}
We proceed as in the proof of Proposition~\ref{prop:compact-choice} with $X=[0,1]$.
By Lemma~\ref{lem:completion-totalization} and analogously to Corollary~\ref{cor:completion-choice} it suffices to prove $\T\ConC_{[0,1]}\leqSW\ConC_{[0,1]}$. 
For $m=\langle k,\langle n_0,...,n_k\rangle\rangle\in\IN$ we define $l(m):=\sup\{x\in[0,1]:[0,x]\In B_{n_0}\cup...\cup B_{n_k}\}$
and $r(m):=\inf\{y\in[0,1]:[y,1]\In B_{n_0}\cup...\cup B_{n_k}\}$. 
Here we assume $l(m):=0$ and $r(m):=1$ if the respective sets are empty.
Then the set
\[B:=\{m\in\IN:l(m)\leq r(m)\}\]
is computable, as the values $l(m)$ and $r(m)$ can be computed as rational numbers.
Hence, given a list $p\in\IN^\IN$ of balls $B_{p(i)}$ with $A=X\setminus\bigcup_{i\in\IN}B_{p(i)}$ we generate
a list $q$ of numbers of open rational intervals $[0,l(m))$ and $(r(m),1]$ with $m=\langle i,\langle p(0),...,p(i)\rangle\rangle$ 
as long as $l(m)\leq r(m)$ and we indefinitely continue with the last rational intervals with this property if eventually $l(m)>r(m)$ (which means that $A=\emptyset$).
Due to compactness of $[0,1]$ it is guaranteed that $q$ is a name of the set $A$, if this set $A$ is a non-empty closed and connected set
and it is a name of some other non-empty closed and connected set, otherwise. That is 
The map $p\mapsto q$ is a computable realizer for a multi-valued retraction $r:\AA_-([0,1])\mto\dom(\ConC_{[0,1]})$ 
with $r(A)=A$ for non-empty connected $A\In[0,1]$, and $r(A)$ is some non-empty connected subset of $[0,1]$ otherwise.
Such a retraction is all what is needed to prove $\T\ConC_{[0,1]}\leqSW\ConC_{[0,1]}$.
\end{proof}

Some important choice problems are also co-complete.
We prove a rather technical but fairly general result about restrictions of choice first.

\begin{proposition}
\label{prop:C-co-total}
Let $D\In\AA_-[0,1]$ be such that $[a,b]\in D$ and $\C_{[0,1]}|_D\leqW\C_{[a,b]}|_D$ for all $[a,b]\In[0,1]$ with $a<b$.
Then $\C_{[0,1]}|_D$ is co-complete and co-total.
\end{proposition}
\begin{proof}
Without loss of generality, we assume that $D$ contains only non-empty sets.
We prove that $\C_{[0,1]}|_D$ is co-total. By Corollary~\ref{cor:co-completeness-totality} it follows that it is also co-complete.
Let $g:\In (X,\delta_X)\mto (Y,\delta_Y)$ be some problem.
We assume that $\C_{[0,1]}|_D\leqW\T g$ is witnessed by computable $H,K:\In\IN^\IN\to\IN^\IN$. 
Let $G\vdash\T g$ and let $p$ be a name of $[0,1]$.
Then $H\langle p,GK(p)\rangle$ determines a real number $x$ with precision $\varepsilon<\frac{1}{3}$ after reading only a finite prefix of $w\prefix p$.
We now consider the set
\[A:=\{J\in D:(\exists p\in w\IN^\IN)(\delta_{\AA_-([0,1])}(p)=J\mbox{ and }\delta_XK(p)\not\in\dom(g))\}.\]
We claim that there is some $[a,b]\In[0,1]$ with $a<b$ and such that for all $J\In[a,b]$ with $J\in D$ we have $J\not\in A$.
Let us assume the contrary. Then for $I_0:=[0,\frac{1}{3}]$ and $I_1:=[\frac{2}{3},1]$ there are $J_i\In I_i$ with $J_i\in D$ such that $J_i\in A$ for $i\in\{0,1\}$.
This implies that there are names $p_i\in w\IN^\IN$ of $J_i$ for $i\in\{0,1\}$ such that $K(p_i)$ is not a name of a point in $\dom(g)$.
Hence, there is a realizer $G_1$ of $\T g$ with $GK(p)=G_1K(p)=G_1K(p_i)$ for $i\in\{0,1\}$.
This is a contradiction since the distance between $I_0$ and $I_1$ is $\frac{1}{3}$.
Hence, we have proved the claim and there is a $[a,b]\In[0,1]$ with the desired properties.
That means that $K,H$ also witness $\C_{[a,b]}|_D\leqW g$, where we use that $\delta_{\AA_-([0,1])}|_{w\IN^\IN}\equiv\delta_{\AA_-([0,1])}$.
This implies $\C_{[0,1]}|_D\leqW\C_{[a,b]}|_D\leqW g$, which means that $\C_{[0,1]}|_D$ is co-total. 
\end{proof}

This result can be readily applied to several important variants of choice.
In particular, we obtain the following.
We note that $\C_{2^\IN}\equivSW\C_{[0,1]}$ by \cite[Corollary~4.6]{BBP12}. 

\begin{corollary}
\label{cor:CCI-co-total}
$\C_{2^\IN}$ and $\ConC_{[0,1]}$ are co-complete and co-total.
\end{corollary}

\section{Positive Choice}
\label{sec:positive-choice}

In this section we want to study $\PC_X$, which is $\C_X$ restricted to sets of positive measures.
This requires that we have a fixed Borel measure on $X$ and we are mostly interested in the
cases $X=2^\IN$, $X=[0,1]$ and $X=\IR$. In the first case we use the uniform measure $\mu$ and in the
second and third case the Lebesgue measure $\lambda$. It is known that $\PC_{2^\IN}\equivSW\PC_{[0,1]}\equivSW\WWKL$
(see \cite[Proposition~8.2]{BGH15a} for these results and the definition of $\WWKL$).
By $\PCC_X$ we denote the restriction of $\PC_X$ to connected sets.
The following observation is a direct consequence of Proposition~\ref{prop:C-co-total}.

\begin{corollary}
\label{cor:PCCI-co-total}
$\PC_{2^\IN}$ and $\PCC_{[0,1]}$ are co-complete and co-total.
\end{corollary}

The following result allows us to show that neither $\PC_{2^\IN}$ nor $\PCC_{[0,1]}$ are complete.

\begin{proposition}
\label{prop:PCC-PC}
$\overline{\PCC_{[0,1]}}\nleqW\PC_{[0,1]}$.
\end{proposition}
\begin{proof}
We consider the problem 
\[P:\AA_-[0,1]\mto\overline{[0,1]},A\mapsto\left\{\begin{array}{ll}
A & \mbox{if $A=[a,b]\In[0,1]$ with $a<b$}\\
\overline{[0,1]} & \mbox{otherwise}
\end{array}\right.\]
Clearly, $P\leqW\overline{\PCC_{[0,1]}}$ and hence it suffices to show $P\nleqW\PC_{[0,1]}$.
In \cite[Proposition~15.1]{BGH15a} we proved $\ConC_{[0,1]}\nleqW\PC_{[0,1]}$ and literally
the same proof can be used to show $P\nleqW\PC_{[0,1]}$.
This is because the proof does only exploit the values of $\ConC_{[0,1]}$ on non-singleton intervals
and the fact that $\ConC_{[0,1]}$ is also somehow defined on singletons. In both respects, $P$ behaves
like $\ConC_{[0,1]}$. The fact that the output is considered on $\overline{[0,1]}$ instead of $[0,1]$
also causes no changes, since we only exploit outputs that are actually in $[0,1]$.
\end{proof}

As a consequence of this result we obtain that positive choice is actually not complete.

\begin{corollary}
\label{cor:PC-PCC}
$\PC_{2^\IN}\lW\overline{\PC_{2^\IN}}$ and $\PCC_{[0,1]}\lW\overline{\PCC_{[0,1]}}$.
\end{corollary}

The problems $f\leqW\PC_{2^\IN}$ have been characterized in \cite{BGH15a} exactly as the {\em Las Vegas computable} problems $f$.
Hence we obtain the following by Corollary~\ref{cor:completion-total}.

\begin{corollary}[Las Vegas computability]
\label{cor:Las-Vegas}
Las Vegas computability is not preserved downwards by (strong) total Weihrauch reducibility.
\end{corollary}

Now we want to prove that $\overline{\PCC_{[0,1]}}\nleqW\PC_\IR$ holds. 
For this purpose it is useful to use fractality as a property.
We recall that a problem $f$ is called a {\em fractal}~\cite{BGM12}, if there is a problem $F:\In\IN^\IN\mto\IN^\IN$
such that $F\equivW f$ and $F|_A\equivW F$ holds for every clopen $A\In\IN^\IN$ with $A\cap\dom(F)\not=\emptyset$.
If $F$ can be chosen to be total, then $f$ is called a {\em total fractal} and if we can replace $\equivW$ by $\equivSW$, then
we speak of a strong (total) fractal. In \cite[Lemma~15.5]{BGH15a} it was proved that $\ConC_{[0,1]}$ is a total fractal.
We follow the lines of that proof to obtain the following result.

\begin{lemma}
\label{lem:PCC-strong-fractal}
$\overline{\PCC_{[0,1]}}$ is a strong total fractal.
\end{lemma}
\begin{proof}
In \cite[Proposition~3.6]{BG11a} it was proved that $\PCC_{[0,1]}\equivSW\B_I^-$, where
\[\B_I^-:\In\IR_<\times\IR_>\to\IR,(a,b)\mapsto[a,b]\]
with $\dom(\B_I^-):=\{(a,b)\in\IR^2:a<b\}$. ($\PCC_{[0,1]}$ was called $\C_I^-$ in \cite{BG11a}.)
Here $\IR_<$ and $\IR_>$ are represented by representations $\rho_<$ and $\rho_>$
as limits of increasing and decreasing sequences of rational number, respectively~\cite{Wei00}. 
We consider the problem $G:\In\IN^\IN\mto\IN^\IN$ that maps every name of a pair $(a,b)\in\IR_<\times\IR_>$
with $a<b$ to any name of any point $y\in\IR$ with $a\leq y\leq b$ and that is undefined for other inputs. 
Then $G=(\B_I^-)^\r\equivSW\B_I^-$.
There is a computable function $K:\IN^\IN\to\IN^\IN$ that maps every pair $\langle p,q\rangle\in\IN^\IN$,
where $p,q\in\IN^\IN$ are interpreted as sequences $(a_n)_n$ and $(b_n)_n$ of rational numbers, to a pair $\langle p',q'\rangle$ that satisfies the following
conditions: $p'$ and $q'$ encode increasing and decreasing sequences $(c_n)_n$ and $(d_n)_n$ of rational numbers, respectively, with $c_n<d_n$
and if $(a_n)_n$ and $(b_n)_n$ are also increasing and decreasing, respectively, with $a_n<b_n$ and $\sup_{n\in\IN}a_n\leq\inf_{n\in\IN}b_n$, then 
$c_n=a_n$ and $d_n=b_n$ for all $n\in\IN$. Such a computable $K$ can be realized by going through the sequences $(a_n)_n$ and $(b_n)_n$
and as long as $a_0\leq a_1\leq...\leq a_k$ and $b_0\geq b_1\geq ...\geq b_k$ and $a_k<b_k$ we choose $c_i:=a_i$ and $d_i:=b_i$ for $i=0,...,k$
and as soon as one of the conditions is violated, we just continue with the last consistent pair (in the case that there is no such pair, we use $c_i:=0$ and $d_i:=1$).
Then $F:=GK:\In\IN^\IN\mto\IN^\IN$ is an extension of $G$, which is only undefined if the input is a name of a pair $(a,b)\in\IR_<\times\IR_>$ with $a=b$. 
We also have $F\equivSW G\equivSW\B_I^-\equivSW\PCC_{[0,1]}$.
Hence, it suffices to show that $\overline{F}: \overline{\IN^\IN}\mto\overline{\IN^\IN}$ is a strong total fractal.
In fact, we claim that $\overline{F}^\r\leqSW\overline{F}^r|_A$ for every clopen $A:=w\IN^\IN\In\IN^\IN$.
Let $u:=w-1$. Then $u$ determines a finite prefix $v$ of $K(u\IN^\IN)$ of the same length as $u$ and this
prefix encodes a rational interval $[a,b]$ with $a<b$; if $u=v=\varepsilon$, then we assume $[a,b]:=[0,1]$.
Now we can use a computable bijective map $T:\IR\to(a,b)$ and its computable inverse $T^{-1}$ to reduce $\overline{F}^\r$ to $\overline{F}^\r|_A$.
Hence, $\overline{F}^\r$ and $\overline{F}$ are strong total fractals.
\end{proof}

Using this lemma we can apply a choice elimination result by Le Roux and Pauly~\cite[Theorem~2.4]{LRP15a} to obtain the following corollary. 
We use the {\em compositional product} of problems defined by $f*g:=\max_{\leqW}\{f_0\circ g_0:f_0\leqW f,g_0\leqW g\}$ \cite{BGM12,BP18}.

\begin{corollary}
\label{cor:PCC-PCR}
$\overline{\PCC_{[0,1]}}\nleqW\PC_\IR$.
\end{corollary}
\begin{proof}
By~\cite[Corollary~6.4, Proposition~7.4]{BGH15a} we have $\PC_\IR\equivW\PC_{2^\IN}*\C_\IN$. 
Hence, $\overline{\PCC_{[0,1]}}\leqW\PC_\IR$ would imply $\overline{\PCC_{[0,1]}}\leqW\PC_{2^\IN}$ by \cite[Theorem~2.4]{LRP15a}, since $\overline{\PCC_{[0,1]}}$ is a total fractal by Lemma~\ref{lem:PCC-strong-fractal}.
This contradicts Proposition~\ref{prop:PCC-PC}. 
\end{proof}

This implies in particular $\overline{\PC_{2^\IN}}\nleqW\PC_\IR$.
In \cite[Proposition~3.8]{BG11a} it was proved that $\PCC_{[0,1]}\leqW\C_\IN$ holds.
This implies that $\PCC_{[0,1]}$ is not a total fractal (since otherwise $\PCC_{[0,1]}\leqW\id$ would follow by \cite[Theorem~2.4]{LRP15a},
which is incorrect as $\PCC_{[0,1]}$ is not computable).
The cited reduction also implies $\overline{\PCC_{[0,1]}}\leqW\overline{\C_\IN}$.
However, by Lemma~\ref{lem:PCC-strong-fractal} and \cite[Theorem~2.4]{LRP15a} we obtain the following conclusion.

\begin{corollary}
\label{cor:PCC-CN}
$\overline{\PCC_{[0,1]}}\nleqW\C_\IN$.
\end{corollary}

In order to get some upper bounds on $\overline{\PC_{2^\IN}}$ and in order to separate it from $\C_{2^\IN}$ it is useful to consider
the {\em negligibility problem}, i.e., the characteristic function of 
sets of measure zero:
\[\NEG:\AA_-(2^\IN)\to\{0,1\},A\mapsto\left\{\begin{array}{ll}
   1 & \mbox{if $\mu(A)=0$}\\
   0 & \mbox{otherwise}
\end{array}\right..\]

It is easy to see that the negligibility problem is equivalent to $\LPO'$.

\begin{lemma}[Negligibility]
\label{lem:negligibility}
$\LPO'\equivSW\NEG$.
\end{lemma}
\begin{proof}
We note that $\mu:\AA_-(2^\IN)\to\IR_>$ is computable (see also \cite[Lemma~2.7]{BGH15a}), where $\IR_>$ denotes the set of upper reals
that are represented as an infimum of a decreasing sequence of rational numbers. 
Since the identity $\iota:\IR_>\to\IR,x\mapsto x$ is limit computable, i.e., $\iota\leqSW\lim$, 
and $\LPO$ can be used to decide equality on the reals,
we obtain a computable $K:\AA_-(2^\IN)\mto\IN^\IN$ such that $\NEG=\LPO\circ\lim\circ K$.
This proves $\NEG\leqSW\LPO'$. 

For the other direction we use Lemma~\ref{lem:infinity}
and we prove $\INF\leqSW\NEG$.
Given a sequence $p\in\IN^\IN$, we want to find out whether there are infinitely many $n\in\IN$ with $p(n)=0$.
Hence we compute a name $K(p)$ of the set $A$ with $A=2^\IN\setminus\bigcup_{i=0}^k 1^i0\IN^\IN$, provided that we find $k\in\IN\cup\{\infty\}$ many zeros in $p$.
Then $\mu(A)=0\iff k=\infty\iff \INF(p)=1$. 
Hence $\INF=\NEG\circ\delta_{\AA_-(2^\IN)}\circ K$, which proves $\INF\leqSW\NEG$.
\end{proof}

The negligibility problem can be used to reduce $\T\PC_{2^\IN}$ to $\PC_{2^\IN}''$ and in consequence
to separate $\overline{\PC_{2^\IN}}$ from $\C_{2^\IN}$.

\begin{corollary}[Positive choice]
$\PC_{2^\IN}\lW\overline{\PC_{2^\IN}}\leqW\T\PC_{2^\IN}\lW\C_{2^\IN}$ and we have $\T\PC_{2^\IN}\leqW\PC_{2^\IN}''$.
\end{corollary}
\begin{proof}
By Corollary~\ref{cor:PC-PCC} we have $\PC_{2^\IN}\lW\overline{\PC_{2^\IN}}$.
By the remark after Corollary~\ref{cor:completion-choice}
we obtain $\overline{\PC_{2^\IN}}\leqW\T\PC_{2^\IN}$.
Clearly $\T\PC_{2^\IN}\leqW\PC_{2^\IN}''$, as $\NEG$ can be used to decide whether the input of $ \T\PC_{2^\IN}$ is in its domain, and 
Lemma~\ref{lem:negligibility} implies that $\NEG\leqW\LPO'\leqW\lim'$. 
By \cite[Corollary~14.9]{BGH15a} it is known that $\C_{2^\IN}\nleqW\PC_{2^\IN}''$. Hence,
$\T\PC_{2^\IN}\lW\T\C_{2^\IN}\equivW\C_{2^\IN}$ by Proposition~\ref{prop:compact-choice}.
\end{proof}

In particular, this result shows that $\T\PC_{2^\IN}$ and $\overline{\PC_{2^\IN}}$ are probabilistic 
in the sense defined in \cite{BGH15a}.

\section{Choice on the Natural Numbers}
\label{sec:choice-natural}

In this section we study choice on natural numbers.
Since $\lim_\IN\equivSW\C_\IN$, we get the following conclusion from Corollary~\ref{cor:co-complete-co-total}.

\begin{corollary}
\label{cor:CN-co-total-complete}
$\C_\IN$ is co-complete and co-total.
\end{corollary}

On the other hand, $\C_\IN$ is not complete.
Since it is known by \cite[Theorem~7.12]{BGM12} that $f\leqW\C_\IN$ holds if and only if $f$ is computable with finitely
many mind changes, it suffices to show that $\overline{\C_\IN}$ is not computable with finitely
many mind changes in order to conclude that $\C_\IN\lW\overline{\C_\IN}$ holds.

\begin{proposition}[Choice on natural numbers]
\label{prop:choice-natural}
$\overline{\C_\IN}$ is limit computable and not computable with finitely many mind changes, and $\T\C_\IN$ is not even limit computable. 
\end{proposition}
\begin{proof}
$\C_\IN$ is computable with finitely many mind changes and hence, in particular, limit computable. 
By Corollary~\ref{cor:completion-total} we have $\C_\IN\equivSTW\overline{\C_\IN}$.
Since limit computability is preserved downwards by total Weihrauch reducibility \cite[Proposition~4.9]{BG20},
it follows that $\overline{\C_\IN}$ is limit computable. 
We prove that $\overline{\C_\IN}$ is not computable with finitely many mind changes.
This implies that $\T\C_\IN$ is also not computable with finitely many mind changes by Corollary~\ref{cor:completion-choice}.
Since the output space of $\T\C_\IN$ is $\IN$, this implies that $\T\C_\IN$ is not even limit computable by \cite[Proposition~13.10]{BGM12}.
Let us assume the contrary and let us consider a Turing machine that computes
$\overline{\C_\IN}$ with finitely many mind changes. 
Upon input of a name of $\IN\in\overline{\AA_-(\IN)}$, the machine
eventually has to produce a natural number $n_0$ as output after seeing only a finite
prefix of the input. After this finite prefix the input can be modified to a name
of the set $\IN\setminus\{n_0\}$, in which case the machine has to change its mind and produce
a new output $n_1\not=n_0$ after seeing a longer prefix of the input. Now one can change
the input to an input of $\IN\setminus\{n_0,n_1\}$, in which case the machine has to change
its mind again and it has to produce an output $n_2\not\in\{n_0,n_1\}$.
This process can be continued inductively and it produces a name of co-infinite (possibly empty) set $A\in\overline{\AA_-(\IN)}$
upon which the given machine has to change its mind infinitely often. Since $A\in\dom(\overline{\C_\IN})$,
the machine does not operate with finitely many mind changes on a valid input.
\end{proof}

Proposition~\ref{prop:choice-natural} implies that for the space $X=\IN$ the choice principle $\C_X$, its completion and
its totalization lead to three different degrees.

\begin{corollary}[Choice on natural numbers]
\label{cor:choice-natural}
$\C_\IN\lW\overline{\C_\IN}\lW\T\C_\IN$.
\end{corollary}

This also means that finite mind change computability is not preserved downwards by total Weihrauch reducibility
and by a contrapositive version of the reasoning used for the proof of \cite[Proposition~4.9]{BG20}, it follows that finite mind change
computability does not respect precompleteness.

\begin{corollary}
\label{cor:finite-mind-change}
Finite mind change computability is not preserved downwards by (strong) total Weihrauch reducibility and does
not respect precompleteness.
\end{corollary}

As $\T\C_\IN\leqW\T\lim$ is easy to see, Propositions~\ref{prop:choice-natural} and \ref{prop:complete-problems}
imply the following.

\begin{corollary}
$\lim\equivSW\overline{\lim}\lW\T\lim$.
\end{corollary}

Since $\lim\equivSW\J$ by \cite[Lemma~8.9]{BBP12} and $\T\J=\J$, this also proves that $f\mapsto \T f$ is not a closure operator with respect to $\leqW$.

We mention that the given proof of Proposition~\ref{prop:choice-natural} does not change in presence of any oracle. 
Hence, $\overline{\C_\IN}$ is not finite mind change computable with respect to any oracle and hence
not even reducible to $\C_\IN$ with respect to the continuous version of Weihrauch reducibility. 
Using a jump inversion property \cite[Theorem~11]{BHK17} and Proposition~\ref{prop:completion-jump-choice}
this yields the following corollary.

\begin{corollary}
\label{cor:CNS}
$\C_\IN'\lW\overline{\C_\IN'}\equivSW{\overline{\C_\IN}\,}'$.
\end{corollary}

Alternatively, we could also prove Proposition~\ref{prop:choice-natural} by showing that $\overline{\C_\IN}$ is a total fractal.
This fact is useful by itself and will be used later.

\begin{lemma}
\label{lem:CN-strong-fractal}
$\overline{\C_\IN}$ is a strong total fractal.
\end{lemma}
\begin{proof}
We consider the representation $\delta$ of $\AA_-(\IN)$ given by $\delta(p):=\IN\setminus\range(p-1)$.
It is easy to see that we have $\delta\equiv\delta_{\AA_-(\IN)}$.
We consider $F:\IN^\IN\mto\IN^\IN$ defined by 
\[F(p):=\left\{\begin{array}{ll}
(\delta_\IN^\wp)^{-1}\circ\C_\IN\circ\delta^\wp(p) & \mbox{if $p\in\dom(\C_\IN\circ\delta^\wp)$}\\
\IN^\IN & \mbox{otherwise}
\end{array}\right.\]
As explained in the proof of Lemma~\ref{lem:completion-totalization-realizer} we have $F\equivSW\overline{\C_\IN}$.
If $w:=a_0...a_k\in\IN^*$ and hence $A:=w\IN^\IN$ is a clopen set, then we can also prove that $F\leqSW F|_A$. We define $K(p)$ such 
that the prefix $w$ is filled up with a word $v$ that contains all numbers up to $m:=\max\{a_0,...,a_k\}+1$ and then for all numbers $i\geq2$ in $\range(p)$ the number $i+m-1$ is added,
i.e., $K(p)=wvq$ with $q(n):=p(n)+m-1$ if $p(n)\geq2$ and $q(n)=p(n)$ if $p(n)\leq1$. Together with a computable function $H$ with $H(r)(n):=\max(0,r(n)-m+1)$ the function $K$ witnesses $F\leqSW F|_A$. 
\end{proof}

It is worth noting that there is also a specific interesting problem below $\overline{\C_\IN}$
that is not below $\C_\IN$. We recall that the {\em Bolzano-Weierstra\ss{} theorem} on the 
two point space $\{0,1\}$ is defined by {$\BWT_2:2^\IN\mto\{0,1\},p\mapsto\{i:i$ is a cluster point of $p\}$}.
This problem was studied in \cite{BGM12}. Above we have already introduced the
weak version $\WBWT_2:2^\IN\mto 2^\IN$ of it that has been studied in \cite{BHK17a}.

\begin{proposition}[Weak Bolzano-Weierstra\ss{} theorem]
\label{prop:WBWT}
We have $\WBWT_2\leqW\overline{\C_\IN}$ and $\WBWT_2\nleqW\C_\IN$.
\end{proposition}
\begin{proof}
Given a binary sequence $p\in 2^\IN$, we generate a list $K(p)$ of all natural numbers $n=0,1,2,...$
as long as we see zeros in $p$. Whenever we see a one in $p$, then we repeat the previous digit on the output (or zero, if no output has been written yet).
That is we compute a name $K(p)$ of a set $A\In\IN$ which is empty if and
only if $p$ contains infinitely many zeros. Given a point $n\in\overline{\C_\IN}(A)$ together
with $p$ we try to check whether the set represented by $K(p)$ contains $n$. As long as $n$ has not been removed from this set, we build an infinite
binary sequence $q\in 2^\IN$ that consists of digits $1$. In the moment where we find that $n$ is removed from the set represented by $K(p)$,
we change to producing digits $0$. That will happen if and only if $A$ is empty, i.e., if and only if $p$ contains infinitely many zeros.
In this case the output is of the form $q=1^k\widehat{0}$. In the case that $A$ is not empty, $q=\widehat{1}$.
In any case, we have that $\lim_{i\to\infty}q(i)$ is a cluster point of $p$, i.e., $q\in\BWT_2(p)$.
This proves $\WBWT_2\leqW\overline{\C_\IN}$.

It is easy to see that $\WBWT_2\nleqW\C_\IN$, as $\BWT_2\leqW\lim_\IN*\WBWT_2$ and $\BWT_2$ is not limit computable by~\cite[Proposition~12.5]{BGM12},
whereas $\lim_\IN*\C_\IN\equivW\C_\IN$ by \cite[Proposition~3.8]{BGM12} and \cite[Corollary~7.6]{BBP12} and $\C_\IN$ is limit computable.
\end{proof}

As a direct corollary we can conclude that $\WBWT_2$ is not co-complete.

\begin{corollary}
$\WBWT_2$ is not co-complete.
\end{corollary}

There are also specific interesting problems that are below $\T\C_\IN$ and not below $\overline{\C_\IN}$, and problems that are
below $\C_\IN*\overline{\C_\IN}$ and not below $\T\C_\IN$.
The problem $\C_\IN*\overline{\C_\IN}$  belongs to a class that is interesting by itself and that was already studied by 
Neumann and Pauly~\cite{NP18}.

\begin{corollary}
\label{cor:CN-!CN}
$\C_\IN*\overline{\C_\IN}\equivW\C_\IN*\T\C_\IN\equivW\C_\IN*\SORT\equivW\C_\IN*\LPO'$.
\end{corollary}
\begin{proof}
The equivalences $\C_\IN*\T\C_\IN\equivW\C_\IN*\SORT\equivW\C_\IN*\INF_\IS$ were proved in \cite[Corollary~30]{NP18}.
We note that $\INF_\IS\equivW\LPO_\IS'$ by Lemma~\ref{lem:infinity} and $\LPO'\leqW\LPO*\LPO_\IS'\leqW\C_\IN*\INF_\IS$. 
This proves $\C_\IN*\LPO'\equivW\C_\IN*\INF_\IS$, as $\LPO_\IS\leqSW\LPO$ and $\C_\IN*\C_\IN\equivW\C_\IN$.
It is clear that $\overline{\C_\IN}\leqW\T\C_\IN$ implies $\C_\IN*\overline{\C_\IN}\leqW\C_\IN*\T\C_\IN$.
By Proposition~\ref{prop:retraction-N} there is a retraction $\overline{\IN}\to\IN$ that is computable with finitely many 
mind changes and by Corollary~\ref{cor:completion} there is a computable injection $\iota:\AA_-(\IN)\to\overline{\AA_-(\IN)}$.
This implies $\T\C_\IN\leqW\C_\IN*\overline{\C_\IN}$ and hence $\C_\IN*\T\C_\IN\equivW\C_\IN*\overline{\C_\IN}$ as $\C_\IN*\C_\IN\equivW\C_\IN$.
\end{proof}

As a corollary we obtain the following.

\begin{corollary}
\label{cor:LPO'-CN} We obtain
\begin{enumerate}
\item  $\LPO'\leqW\C_\IN*\overline{\C_\IN}$, but $\LPO'\nleqW\T\C_\IN$.
\item $\LPO_\IS'\leqW\T\C_\IN$, but $\LPO_\IS'\nleqW\overline{\C_\IN}$.
\end{enumerate}
\end{corollary}
\begin{proof}
$\LPO'\leqW\C_\IN*\overline{\C_\IN}$ holds by Corollary~\ref{cor:CN-!CN}.
Since $\LPO'$ is co-total by Corollary~\ref{cor:co-complete-co-total}, it follows that $\LPO'\nleqW\T\C_\IN$ holds,
since otherwise $\LPO'\leqW\C_\IN$ would follow, which is false as $\LPO'$ is not limit computable by Lemma~\ref{lem:LPO-lim}.
$\LPO_\IS'\leqW\T\C_\IN$ was proved in \cite[Proposition~24]{NP18}.
Since $\LPO_\IS'$ is co-complete by Corollary~\ref{cor:co-complete-co-total}, it follows that $\LPO_\IS'\nleqW\overline{\C_\IN}$ holds,
since otherwise $\LPO_\IS'\leqW\C_\IN$ would follow, which is false as $\LPO'\leqW\C_\IN*\LPO_\IS'$.
\end{proof}

As a direct corollary we obtain that $\LPO_\IS'$ is not co-total.

\begin{corollary}
\label{cor:LPOS-co-total}
$\LPO_\IS'$ is not co-total.
\end{corollary}

Completeness can also be used as a separation tool, as we illustrate with the following result.
Basically, the point is that an incomplete problem cannot be factorized into complete problems.
We use that $\C_\IN'\equivW\C_\IN'*\C_\IN'$ holds; this is easy to see and was already stated in the proof of \cite[Proposition~21]{BHK17}.

\begin{proposition}
\label{prop:KN-CN}
$\K_\IN'\lW\K_\IN'*\K_\IN'\lW\C_\IN'$ and $\C_\IN*\overline{\C_\IN}\leqW\K_\IN'*\K_\IN'$.
\end{proposition}
\begin{proof}
The first reduction is obvious, the second reduction follows as $\K_\IN\leqSW\C_\IN$ and $\C_\IN'*\C_\IN'\equivW\C_\IN'$.
The reduction $\C_\IN*\overline{\C_\IN}\leqW\K_\IN'*\K_\IN'$ holds, as 
$\C_\IN\leqW\K_\IN'$ by \cite[Proposition~7.2]{BR17}, 
completion is a closure operator
and $\K_\IN'$ is complete by Corollary~\ref{cor:KN}.
Let us assume that $\K_\IN'*\K_\IN'\leqW\K_\IN'$ holds, then $\SORT\leqW\C_\IN*\SORT\leq\C_\IN*\overline{\C_\IN}\leqW\K_\IN'$ follows by Corollary~\ref{cor:CN-!CN},
in contradiction to \cite[Propositions~16 and 20]{BHK17}. Hence, the first reduction is strict.
The second reduction is strict, as $\K_\IN'$ is complete by Corollary~\ref{cor:KN}, $\C_\IN'$ is not complete by Corollary~\ref{cor:CNS}
and the compositional product preserves completeness by \cite[Proposition~7.6]{BG20}.
\end{proof}

Altogether we obtain the following reduction chain for some of the discussed classes.

\begin{corollary}
\label{cor:TCN-CNS}
$\T\C_\IN\lW\C_\IN*\overline{\C_\IN}\lW\C_\IN'$.
\end{corollary}
\begin{proof}
The reduction $\T\C_\IN\leqW\C_\IN*\overline{\C_\IN}$ was proved in the proof of Corollary~\ref{cor:CN-!CN}
as well as $\C_\IN*\overline{\C_\IN}\equivW\C_\IN*\LPO'$.
Since $\LPO'\leqW\C_\IN'$ and $\C_\IN'\equivW\C_\IN'*\C_\IN'$, we obtain $\C_\IN*\LPO'\leqW\C_\IN*\C_\IN'\equivW\C_\IN'$.
The strictness of the first reduction follows from Corollary~\ref{cor:LPO'-CN} and the strictness of the second reduction from Proposition~\ref{prop:KN-CN}.
\end{proof}

Since $\C_\IN$ is not complete, the cone below $\C_\IN$ in the Weihrauch lattice differs from the cone below $\overline{\C_\IN}$ and hence
it is important to check how this impacts on separation results. 
An important separation is $\ConC_{[0,1]}\nleqW\C_\IN$~\cite[Proposition~4.9]{BG11a}.
This was strengthened to $\ConC_{[0,1]}\nleqW\K_\IN'$ in \cite[Proposition~20]{BHK17}.
With the help of Corollary~\ref{cor:CCI-co-total} we can strengthen the separation in another direction.

\begin{corollary}
\label{cor:CCI-TCN}
$\ConC_{[0,1]}\nW\T\C_\IN$.
\end{corollary}

Here $\T\C_\IN\nleqW\ConC_{[0,1]}$ follows since $\ConC_{[0,1]}$ is limit computable and
$\T\C_\IN$ is not by Proposition~\ref{prop:choice-natural}. 
Since $\ConC_{[0,1]}\leqW\SORT$ by \cite[Proposition~16]{BHK17}, we obtain $\SORT\nleqW\T\C_\IN$.
Since $\SORT$ is also limit computable, we obtain $\T\C_\IN\nleqW\SORT$ and $\SORT\nW\T\C_\IN$.
This was also proved by Neumann and Pauly~\cite[Proposition~24]{NP18}.
In this context it is interesting to note that $\overline{\C_\IN}\leqW\SORT$ holds.

\begin{corollary}
\label{cor:CN-SORT}
$\overline{\C_\IN}\leqSW\SORT$.
\end{corollary}
\begin{proof}
Neumann and Pauly~\cite[Proposition~24]{NP18} proved $\C_\IN\leqW\SORT$ and the proof even
shows $\C_\IN\leqSW\SORT$ (see also \cite[Proposition~12]{BHK17}).
Since completion is a closure operator and by Proposition~\ref{prop:complete-problems} we obtain $\overline{\C_\IN}\leqSW\overline{\SORT}\equivSW\SORT$.
\end{proof}

\section{Lowness}
\label{sec:lowness}

Proposition~\ref{prop:retraction-N} can also be used to prove that $\overline{\C_\IN}$ is not low.
We recall that a problem $f:\In X\mto Y$ is called {\em low},
if it has a realizer of the form $F=\Low\circ G$ with some computable $G:\In\IN^\IN\to\IN^\IN$ and $\Low:=\J^{-1}\circ\lim$.
Lowness was studied, for instance, in \cite{BBP12,Bra18}. By \cite[Theorem~8.10]{BBP12} $f$ is low if and only if $f\leqSW\Low$.
Likewise, $f$ is called low$_2$, if $f\leqSW\Low_2$, where $\Low_2:=\J^{-1}\circ\J^{-1}\circ\lim\circ\lim$.

\begin{corollary}
\label{cor:CN-lowness}
$\overline{\C_\IN}$ is low$_2$ but not low.
\end{corollary}
\begin{proof}
We first prove that $\overline{\C_\IN}$ is not low.
By Proposition~\ref{prop:retraction-N} there is a retraction $r:\overline{\IN}\to\IN$ that is computable with finitely many mind changes and together with
Corollary~\ref{cor:completion} we obtain $\T\C_\IN\leqSW r\circ\overline{\C_\IN}$.
Since $r$ is computable with finitely many mind changes, it is in particular limit computable, and since $\T\C_\IN$ is not limit computable by Proposition~\ref{prop:choice-natural},
it follows that $\overline{\C_\IN}$ cannot be low since the composition of a limit computable problem with a low problem is limit computable by \cite[Corollary~8.16]{BBP12}.

Neumann and Pauly~\cite[Corollary~32]{NP18} proved $\lim*\lim*\SORT\equivW\lim*\lim$, and since $\lim*\lim\equivW\lim'$ is a cylinder,
we obtain $\lim'*\SORT\leqSW\lim'$.
By \cite[Proposition~14.16]{BHK17a} this implies that $\SORT$ is low$_2$.
By Corollary~\ref{cor:CN-SORT} we have $\overline{\C_\IN}\leqSW\SORT$,
which implies $\overline{\C_\IN}$ is also low$_2$.
\end{proof}

By \cite[Corollary~8.14]{BBP12} we have $\C_\IN\leqSW\Low$ and hence $\overline{\C_\IN}\leqSW\overline{\Low}$.
Hence Corollary~\ref{cor:CN-lowness} implies that $\Low$ is not strongly complete.

\begin{corollary}[Low map]
\label{cor:L}
$\Low\lSW\overline{\Low}$.
\end{corollary}

This in turn implies that lowness is not preserved downwards by total Weihrauch reducibility.

\begin{corollary}
\label{cor:lowness}
Lowness is not preserved downwards by (strong) total Weihrauch reducibility and does not respect precompleteness.
\end{corollary}
\begin{proof}
By \cite[Theorem~8.10]{BBP12} $f$ is low if and only $f\leqSW\Low$.
By Corollary~\ref{cor:L} it follows that $\overline{\Low}$ is not low and since $\overline{\Low}\equivSTW\Low$, it follows
that strong total Weihrauch reducibility does not preserve lowness. Hence total Weihrauch reducibility also does
not preserve lowness. By the reasoning used for the proof of \cite[Proposition~4.9]{BG20},
one can conclude that lowness would preserve strong total Weihrauch reducibility if it did respect precompleteness (in the
strong case one does not need closure under juxtaposition with the identity, which is not given for low functions by Lemma~\ref{lem:low-cylinder}).
Hence, lowness does not respect precompleteness. 
\end{proof}

We mention that $\Low$ is not a cylinder. 

\begin{lemma}
\label{lem:low-cylinder}
$\Low$ and $\overline{\Low}$ are not cylinders.
\end{lemma}
\begin{proof}
By \cite[Theorem~8.8]{BBP12} we have $\Low\lSW\Low\times\Low=(\Low\times\id)\circ(\id\times\Low)$.
By \cite[Proposition~8.16]{BBP12} low problems are closed under composition. Hence $\id\times\Low\equivSW\Low\times\id$ cannot
be low, i.e., $\id\times\Low\nleqSW\Low$ and hence $\Low$ is not a cylinder.
By \cite[Proposition~6.19]{BG20} this implies that $\overline{\Low}$ is also not a cylinder.
\end{proof}

Finally, we prove that $\Low$ is not complete, i.e., $\Low\lW\overline{\Low}$.
We even prove a more general result.

\begin{proposition}
\label{prop:WBWT-L}
$\WBWT_2\nleqW\Low$.
\end{proposition}
\begin{proof}
By $\varphi$ we denote a G\"odel numbering of all computable functions ${\varphi_n:\In\IN\to\IN}$.
For every problem $f:\In\IN^\IN\mto Y$ we denote by $f_\varphi:\In\IN\mto Y$ its {\em G\"odelization}
defined by $f_\varphi(n):=f(\varphi_n)$ with $\dom(f_\varphi):=\{n\in\IN:\varphi_n\mbox{ total and }
\varphi_n\in\dom(f)\}$. By the universal Turing machine theorem one obtains $f_\varphi\leqW f$.
Hence, it suffices to show $W:=\WBWT_{2\varphi}\nleqW\Low$ in order to prove our claim.
We prove that $W\leqW\Low$ implies $B:=\BWT_{2\varphi}\leqW\lim$. 
But the latter implies $(\widehat{\BWT_2})_\varphi\leqW\widehat{B}\leqW\lim$, where the first reduction
holds by the smn-theorem. But this is a contradiction,
since it is known that $\widehat{\BWT_2}\equivW\WKL'$ \cite[Corollaries~11.6, 11.7 and 11.12]{BGM12} 
and hence by the relativized Kleene tree construction
$\WKL'$ and thus $\widehat{\BWT_2}$ have computable inputs with no limit computable solution.
We now show that $W\leqW\Low$ implies $B\leqW\lim$.
To this end, let $H,K:\In\IN^\IN\to\IN^\IN$ be computable functions such that
$H\langle\id,GK\rangle$ is a realizer for $W$ whenever $G$ is a realizer for $\Low=\J^{-1}\circ\lim$.
Up to extension the only realizer of $\Low$ is $\Low$ itself, hence we can assume $G=\Low$.
By the smn-theorem there are two computable functions $r_0,r_1:\IN\to\IN$ such that
$\J \Low K(\varphi_n)(r_i\langle n,k\rangle)=1$ if and only if
$H\langle\varphi_n,\Low K(\varphi_n)\rangle(m)=i$ for all $m\geq k$.
Intuitively speaking, $r_i$ inspects the outcome of $H\langle\varphi_n,\Low K(\varphi_n)\rangle$ with 
respect to the question whether it is eventually constant with value $i$, which is possible if the Turing jump
of $\Low K(\varphi_n)$ is known, since $q\mapsto H\langle\varphi_n,q\rangle$ is computable uniformly in $n$.
Now $\J \Low=\lim$ and hence the inner reduction function $K$ also witnesses the reduction
$B\leqW\lim$. More precisely, given an input $\langle n,q\rangle$ the corresponding outer reduction function $H'$ only
has to search for some $(i,k)$ with $i\in\{0,1\}$ and $k\in\IN$ such that 
$q(r_i\langle n,k\rangle)=1$ and output (a name of) $i$ in this case.
Such a function $H'$ is clearly computable and satisfies
$H'\langle n,\lim K(\varphi_n)\rangle\in B(n)$. 
\end{proof}

We obtain the following corollary with Proposition~\ref{prop:WBWT} and Corollary~\ref{cor:CN-SORT}.

\begin{corollary}
\label{cor:L-not-complete}
$\Low$ is not complete, $\overline{\C_\IN}\nleqW\Low$ and $\SORT\nleqW\Low$.
\end{corollary}

\section{Choice on Euclidean Space}
\label{sec:choice-Euclidean}

By Lemma~\ref{lem:CN-strong-fractal} $\overline{\C_\IN}$ is a total fractal.
This fact allows us to give a very simple proof of the following result.

\begin{proposition}[Choice on Euclidean space]
\label{prop:choice-euclidean}
$\C_\IR\lW\overline{\C_\IR}$ and $\PC_\IR\lW\overline{\PC_{\IR}}$.
\end{proposition}
\begin{proof}
Let us assume that $\overline{\C_\IR}\leqW\C_\IR$. Then we obtain
\[\overline{\C_\IN}\leqW\overline{\C_\IR}\leqW\C_\IR\equivW\C_{2^\IN}*\C_\IN,\]
where the first reduction holds since $\C_\IN\leqW\C_\IR$ and completion is a closure operator and the last mentioned equivalence is known~\cite[Example~4.4~(2)]{BGM12}.
By the choice elimination principle~\cite[Theorem~2.4]{LRP15a} it follows that $\C_\IN\leqW\overline{\C_\IN}\leqW\C_{2^\IN}$, which is known to be false~\cite[Corollary~4.2]{BG11a}. 
$\PC_\IR\lW\overline{\PC_{\IR}}$ can be proved analogously, since $\PC_\IR\equivW\PC_{2^\IN}*\C_\IN$ by \cite[Corollary~6.4, Proposition~7.4]{BGH15a}.
\end{proof}

Analogously, one could also prove $\C_\IN\lW\overline{\C_\IN}$.
As a corollary of the proof of Proposition~\ref{prop:choice-euclidean} we also obtain the following separation,
which is also a consequence of Corollary~\ref{cor:L-not-complete}.

\begin{corollary} 
\label{cor:CN-CR}
$\overline{\C_\IN}\nleqW\C_\IR$.
\end{corollary}

Since $\overline{f\times g}\leqSW\overline{f}\times\overline{g}$ by \cite[Proposition~6.3]{BG20} and using Proposition~\ref{prop:compact-choice} we obtain 
\[\overline{\C_\IR}\leqSW\overline{\C_{2^\IN}\times\C_\IN}\leqSW\overline{\C_{2^\IN}}\times\overline{\C_\IN}\leqSW\C_{2^\IN}\times\overline{\C_\IN}\]
and one could ask whether the inverse reduction holds too.
The following choice and completion elimination principle is quite useful and can be used to prove that this is not so.
It has some similarities to the displacement principle formulated in~\cite[Theorem~8.3, Corollary~8.4]{BLRMP19}
and shows that if the completion of a problem $g$ can compute another problem $f$ together with $\C_2$, then the 
uncompleted problem by itself can already compute $f$.

\begin{proposition}[Depletion]
\label{prop:choice-completion-elimination}
$f\times\C_2\leqW\overline{g}\TO f\leqW g$ holds for all problems $f,g$. 
An analogous property holds for $\leqSW$ instead of $\leqW$ in both instances.
\end{proposition}
\begin{proof}
If $f$ is nowhere defined, then the statement holds obviously. Hence, let $f$ be defined somewhere.
We consider $g:\In X\mto Y$ and $\overline{g}:\overline{X}\mto\overline{Y}$. 
Let $p\in\IN^\IN$ be a name of some point $x\in\dom(f)$
and let $q\in\IN^\IN$ be a computable name of the set $\{0,1\}\in\AA_-(\{0,1\})$. 
Let $f\times\C_2\leqW\overline{g}$ be witnessed by some computable functions $H,K$.
Let us assume that $K\langle p,q\rangle$ is a name of a point outside of $\dom(g)$.
Then a realizer $G$ of $\overline{g}$ can produce any value on this input, for instance, $GK\langle p,q\rangle=\widehat{0}$.
Then $H\langle\langle p,q\rangle,\widehat{0}\rangle=\langle s,t\rangle$ where $s$ is a name of a point in $f(x)$ and $i=t(0)\in\{0,1\}$.
Since $H$ is continuous, there are finite prefixes $w\prefix p$ and $v\prefix q$ that suffice to produce the value $i=t(0)$.
Moreover, there is a name $q'$ of $\{i-1\}$ with $v\prefix q'$, and there is a realizer $G$ of $\overline{g}$ that produces a value $r=GK\langle p,q'\rangle$ that starts with
sufficiently many zeros (which is possible with respect to the representation $\delta_{\overline{Y}}$) such that $H\langle\langle p,q'\rangle,r\rangle=\langle s',t'\rangle$ with the same $i=t'(0)$ as above.
But in this case the result is incorrect. Hence, the assumption was incorrect and $K\langle p,q\rangle$ is always a $\delta_{\overline{X}}$--name of a point in $\dom(g)$ for every name $p$ of a point in $\dom(f)$ and the fixed computable $q$.
Hence $K'$ with $K'(p):=K\langle p,q\rangle-1$ is a name of the same point with respect to $\delta_X$ and $H'\langle p,r\rangle:=\pi_1\circ H\langle\langle p,q\rangle,r-1\rangle$ is a computable function
such that $K'$ and $H'$ witness $f\leqW g$. The proof for $\leqSW$ is analogous.
 \end{proof}

One application of Proposition~\ref{prop:choice-completion-elimination} shows that $\overline{\C_\IN}\times\C_{2^\IN}\leqW\overline{\C_\IR}$
would imply $\overline{\C_\IN}\leqW\C_\IR$, which is false by Corollary~\ref{cor:CN-CR}.
An analogous observation holds for $\PC_\IR$.

\begin{corollary}
\label{cor:CR-product}
$\overline{\C_\IR}\lW\C_{2^\IN}\times\overline{\C_\IN}$ and $\overline{\PC_\IR}\lW\PC_{2^\IN}\times\overline{\C_\IN}$.
\end{corollary}

This corollary provides natural examples of problems such that the product of the respective completions is stronger
than the completion of the products. The existence of such examples was already proved in \cite[Lemma~6.9]{BG20}.
Another conclusion that we can draw from Proposition~\ref{prop:choice-completion-elimination} is
that every incomplete problem above $\C_2$ has a completion that is not idempotent.
We recall that a problem $f$ is called {\em idempotent} if $f\times f\equivW f$ holds.

\begin{corollary}[Idempotency and completeness]
\label{cor:idempotency-completeness}
If $f$ is incomplete and $\C_2\leqW f$, then $\overline{f}$ is not idempotent.
\end{corollary}

In particular, this means that our incomplete choice problems are not idempotent.
By \cite[Proposition~6.19]{BG20} a problem that is incomplete has a completion which is not a cylinder.
We recall that a problem $f$ is called a {\em cylinder} if $f\equivSW \id\times f$.

\begin{corollary}
\label{cor:CR-CN-idempotency}
$\overline{\C_\IN}$, $\overline{\C_\IR}$, $\overline{\PCC_{[0,1]}}$, $\overline{\PC_{2^\IR}}$, $\overline{\PC_\IR}$ and $\overline{\C_{\IN^\IN}}$ are not idempotent and not cylinders.
\end{corollary}

Here the statement for $\C_{\IN^\IN}$ already uses $\C_{\IN^\IN}\lW\overline{\C_{\IN^\IN}}$, which is only proved in Theorem~\ref{thm:choice-Baire}.

\section{Choice on Baire Space}
\label{sec:choice-Baire}

Next we want to study the choice problem on Baire space $\C_{\IN^\IN}$.
For this purpose we consider the {\em wellfounded tree problem}, i.e., the characteristic function of the singleton with the empty set as its member:
\[\WFT:\AA_-(\IN^\IN)\to\{0,1\},A\mapsto\left\{\begin{array}{ll}
   1 & \mbox{if $A=\emptyset$}\\
   0 & \mbox{otherwise}
\end{array}\right..\]

By \cite[Theorem~5.2]{BG09} the set $\{\emptyset\}\In\AA_-(\IN^\IN)$ is equivalent to the set of wellfounded trees that is known to be $\Pi^1_1$--complete.
By $\WFT_\IS:\AA_-(\IN^\IN)\to\IS$ we denote the wellfounded tree problem with target space $\IS$.

We start with proving that for every closed set $A\In\IN^\IN$ that is given with respect to the jump representation,
we can compute a closed set $B\In\IN^\IN$ such that $\pi_1(B)=A$. Here $\pi_1:\IN^\IN\to\IN^\IN,\langle p,q\rangle\mapsto p$.

\begin{proposition}[Projections]
\label{prop:projection}
There following problem is computable:
\[P:\AA_-(\IN^\IN)'\mto\AA_-(\IN^\IN), A\mapsto\{B:\pi_1(B)=A\}\]
\end{proposition}
\begin{proof}
Given $A\in\AA_-(\IN^\IN)'$ we can compute by \cite[Proposition~3.6]{BHK18}, \cite[Proposition~9.2]{BGM12} a (possibly empty) sequence $(p_i)_{i\in\IN}$
of points $p_i\in\IN^\IN$ such that $A$ is the set of cluster points of $A$. 
We now start to generate a list of all balls $n\IN^\IN$ with $n\in\IN$ as output while we inspect the sequence $(p_i)_{i\in\IN}$
in stages $i=0,1,2,...$. We say that a $p_i$ is {\em fresh} if it has no common non-empty prefix with any other
previous $p_j$, $j<i$. If, in stage $i$, we encounter some fresh $p_i$, then we select some $k\in\IN$ such that  $\langle p_i(0),k\rangle\IN^\IN$
was not yet enumerated as output and we skip the corresponding ball on the output side (while we continue to enumerate all other balls $n\IN^\IN$).
Additionally we enumerate all balls of the form $\langle p_i(0),k\rangle n\IN^\IN$ with $n\in\IN$. If at some later stage $j$, we encounter some $p_j$ that is
not fresh, but that has a prefix in common with $p_i$ of length greater or equal than $2$, then we select an $l\in\IN$ such that $\langle p_i(0),k\rangle\langle p_i(1),l\rangle\IN^\IN$
was not yet enumerated, we skip this ball on the output side and we additionally start enumerating all balls of the form $\langle p_i(0),k\rangle\langle p_i(1),l\rangle n\IN^\IN$ with $n\in\IN$.
We continue like this inductively. Whenever a non-fresh $p_i$ has a prefix in common with an already enumerated $p_j$ that is longer than the depth of the corresponding
sequence of balls that is enumerated on the output side, then we propagate the corresponding enumeration to the next deeper layer. 
As a result of this, the enumeration of balls on the output side describes the complement of a closed set $B\In\IN^\IN$ such that $\pi_1(B)$ is the set of cluster
points of $(p_i)_{i\in\IN}$, i.e., $\pi_1(B)=A$.
\end{proof}

If $B\in P(A)$ then $B=\emptyset\iff A=\emptyset$. Hence we obtain the following corollary.

\begin{corollary}[Jump of wellfoundedness]
\label{cor:WFT}
$\WFT'\equivSW\WFT$ and $\WFT_\IS'\equivSW\WFT_\IS$.
\end{corollary}

Now we can easily derive the following completeness and co-completeness properties of the wellfoundedness problem.

\begin{corollary}[Wellfoundedness]
\label{cor:wellfoundedness}\ 
\begin{enumerate}
\item $\WFT$ is strongly complete, co-total and co-complete,
\item $\WFT_\IS$ is strongly complete and co-complete.
\end{enumerate}
\end{corollary}
\begin{proof}
We first prove both statements on strong completeness.
The space $\AA_-(\IN^\IN)$ is multi-retraceable by Lemma~\ref{lem:precomplete-closed}.
Hence there is a computable multi-valued retraction $r:\overline{\AA_-(\IN^\IN)}\mto\AA_-(\IN^\IN)$.
By Corollary~\ref{cor:completion} the identity $\iota:\IN^\IN\to\overline{\IN^\IN}$ is a computable embedding.
Altogether this shows that $\overline{\WFT}\leqSW\WFT$, i.e., $\WFT$ is strongly complete.
Also the representation $\delta_\IS$ is easily seen to be precomplete and total and hence $\IS$ is multi-retraceable too by Proposition~\ref{prop:multi-retraceability}.
Hence $\WFT_\IS$ is strongly complete by Corollary~\ref{cor:completeness}.
Co-totality and co-completeness of $\WFT$ and co-completeness of $\WFT_\IS$ follow
from Proposition~\ref{prop:single-valuedness} with the help of Corollary~\ref{cor:WFT}.
\end{proof}

Here we are in particular interested in the co-completeness and co-totality results,
since they help us to establish the following result by an interesting bootstrapping argument.

\begin{proposition}[Wellfoundedness]
\label{prop:wellfoundedness}\
\begin{enumerate}
\item $\WFT_\IS\leqW\T\C_{\IN^\IN}$ and $\WFT_\IS\nleqW\overline{\C_{\IN^\IN}}$, 
\item $\WFT\leqW\C_{\IN^\IN}*\T\C_{\IN^\IN}$ and $\WFT\nleqW\T\C_{\IN^\IN}$.
\end{enumerate}
\end{proposition}
\begin{proof}
It is easy to see that $\WFT_\IS\leqW\T\C_{\IN^\IN}$:
given $A\in\AA_-(\IN^\IN)$ we determine a point $p\in\T\C_{\IN^\IN}(A)$ and we check whether $p\in A$. If not, then we will eventually
recognize that, in which case $A=\emptyset$. If $A\not=\emptyset$, then the search will never terminate, but this is sufficient to compute $\WFT_\IS(A)\in\IS$.
The identity $\iota:\IS\to\{0,1\}$ is easily seen to be equivalent to $\LPO$ and $\WFT=\iota\circ\WFT_\IS$.
Hence the function $f:\IN^\IN\to\{0,1\}$ with $f:=\WFT\circ\delta_{\AA_-(\IN^\IN)}$
satisfies $\WFT\equivW f\leqW\LPO*\WFT_\IS\leqW\C_{\IN^\IN}*\T\C_{\IN^\IN}$. On the other hand, $f$ is not Borel measurable by~\cite[Theorem~5.2]{BG09}
(essentially, since the set of wellfounded trees is $\Pi^1_1$--complete and hence not Borel).
Hence $\WFT\equivW f\nleqW\C_{\IN^\IN}$ by \cite[Theorem~7.7]{BBP12} (this theorem says that the single-valued functions
$g:X\to Y$ on complete computable metric spaces $X,Y$ with $g\leqW\C_{\IN^\IN}$ are exactly the effectively Borel measurable $g$.)
Since $\WFT$ is co-total by Corollary~\ref{cor:wellfoundedness}, $\WFT\nleqW\C_{\IN^\IN}$ implies $\WFT\nleqW\T\C_{\IN^\IN}$.
Now suppose that $\WFT_\IS\leqW\overline{\C_{\IN^\IN}}$. Since $\WFT_\IS$ is co-complete by Corollary~\ref{cor:wellfoundedness},
this would imply $\WFT_\IS\leqW\C_{\IN^\IN}$, which, as above, leads to $\WFT\leqW\LPO*\WFT_\IS\leqW\C_{\IN^\IN}*\C_{\IN^\IN}\equivW\C_{\IN^\IN}$,
where the last equivalence holds by \cite[Corollary~7.6]{BBP12}. Since $\WFT\nleqW\C_{\IN^\IN}$, we obtain $\WFT_\IS\nleqW\overline{\C_{\IN^\IN}}$.
\end{proof}

We can also conclude that $\WFT_\IS$ is not co-total, since otherwise $\WFT_\IS\leqW\C_{\IN^\IN}$ would follow.

\begin{corollary}
\label{cor:WFTS-co-total}
$\WFT_\IS$ is not co-total.
\end{corollary}

Proposition~\ref{prop:wellfoundedness} leads to the following classification of choice problems related to choice on Baire space.

\begin{theorem}[Choice on Baire space]
\label{thm:choice-Baire}
We obtain:\\
$\C_{\IN^\IN}\lW\overline{\C_{\IN^\IN}}\lW\T\C_{\IN^\IN}\lW\C_{\IN^\IN}*\overline{\C_{\IN^\IN}}\equivW\C_{\IN^\IN}*\T\C_{\IN^\IN}$.
\end{theorem}
\begin{proof}
By Corollary~\ref{cor:completion-choice} the reductions $\C_{\IN^\IN}\leqW\overline{\C_{\IN^\IN}}\leqW\T\C_{\IN^\IN}$ are clear
and this implies $\C_{\IN^\IN}*\overline{\C_{\IN^\IN}}\leqW\C_{\IN^\IN}*\T\C_{\IN^\IN}$.
By Proposition~\ref{prop:retraction-N} there is a limit computable retraction $r:\overline{\IN^\IN}\to\IN^\IN$
and $\lim\leqW\C_{\IN^\IN}$ holds by \cite[Example~3.10]{BBP12}.
By Corollary~\ref{cor:completion} the identity $\iota:\AA_-(\IN^\IN)\to\overline{\AA_-(\IN^\IN)}$ is a computable embedding.
Altogether, this implies $\T\C_{\IN^\IN}\leqW\C_{\IN^\IN}*\overline{\C_{\IN^\IN}}$.
Since $\C_{\IN^\IN}\equivW\C_{\IN^\IN}*\C_{\IN^\IN}$ by \cite[Corollary~7.6]{BBP12}, this in turn implies
$\C_{\IN^\IN}*\T\C_{\IN^\IN}\leqW\C_{\IN^\IN}*\C_{\IN^\IN}*\overline{\C_{\IN^\IN}}\equivW\C_{\IN^\IN}*\overline{\C_{\IN^\IN}}$.
The separation statements in $\overline{\C_{\IN^\IN}}\lW\T\C_{\IN^\IN}\lW\C_{\IN^\IN}*\overline{\C_{\IN^\IN}}$ follow from Proposition~\ref{prop:wellfoundedness}.
If we had $\overline{\C_{\IN^\IN}}\leqW\C_{\IN^\IN}$, then $\C_{\IN^\IN}*\overline{\C_{\IN^\IN}}\leqW\C_{\IN^\IN}*\C_{\IN^\IN}\equivW\C_{\IN^\IN}$
would follow, which is not correct.
\end{proof}

By Corollary~\ref{cor:single-valuedness} we obtain the following conclusion on single-valued functions.
We note that for constant $f$ one can easily prove the statement directly.
 
\begin{corollary}[Single-valuedness]
\label{cor:Borel}
Let $X,Y$ be complete computable metric spaces and $f:X\to Y$ a function. Then
$f\leqSW\C_{\IN^\IN}\iff f\leqSW\overline{\C_{\IN^\IN}}\iff f\leqSW\T\C_{\IN^\IN}$.
\end{corollary}

By \cite[Theorem~7.7]{BBP12} the first given condition is exactly satisfied for the effectively Borel measurable functions $f$.

\begin{question}
Can we replace $\leqSW$ by $\leqW$ in Corollary~\ref{cor:Borel}?
\end{question}

\bibliographystyle{plain}
\bibliography{C:/Users/Vasco/Dropbox/Bibliography/lit}
%\bibliography{C:/Users/vbrattka/Dropbox/Bibliography/lit}
%\bibliography{C:/Users/i11avabr/Dropbox/Bibliography/lit}

\end{document}

%% file: BasicChoice.tex
\begin{tikzpicture}[scale=0.8,every node/.style={fill=black!15}]

\node (TCNN) at (1.4,20.75) {$\T\C_{\IN^\IN}$};
\node (!CNN) at (1.4,19.4) {$\overline{\C_{\IN^\IN}}$};
\node (CNN) at (1.4,18.4) {$\C_{\IN^\IN}$};
\node (lim') at (1.4,17) {$\lim'$};
\node (lim) at (1.4,13.4) {$\lim$};
\node (L) at (1.4,10.8) {$\Low$};
\node (!L) at (1.4,11.8) {$\overline{\Low}$};
\node (!CR) at (3.2,11.8) {$\overline{\C_{\IR}}$};
\node (CR) at (3.2,10.8) {$\C_{\IR}$};
\node (!PCR) at (5.2,11.8) {$\overline{\PC_{\IR}}$};
\node (PCR) at (5.2,10.8) {$\PC_{\IR}$};
\node (WWKL) at (7,8.4) {$\PC_{2^\IN}$};
\node (!WWKL) at (7,9.4) {$\overline{\PC_{2^\IN}}$};
\node (PCC) at (9.05,6) {$\PCC_{[0,1]}$};
\node (!PCC) at (9.05,7.2) {$\overline{\PCC_{[0,1]}}$};
\node (WKL) at (3.2,9.4) {$\C_{2^\IN}$};
\node (IVT) at (3.2,7.2) {$\ConC_{[0,1]}$};
\node (SORT) at (9.1,13.4) {$\SORT$};
\node (!CN') at (10.8,17) {$\overline{\C_\IN'}$};
\node (CN') at (10.8,16) {$\C_\IN'$};
\node (TCN) at (10.8,13.4) {$\T\C_\IN$};
\node (!CN) at (10.8,11.8) {$\overline{\C_\IN}$};
\node (CN) at (10.8,10.8) {$\C_\IN$};
\node (KN) at (10.8,6) {$\K_\IN$};
\node(LPO) at (12.4,6) {$\LPO$};
\node(LPOS) at (12.4,13.4) {$\LPO'$};
\node(C2) at (10.8,4.85) {$\C_2$};
\node(C1) at (10.8,3.6) {$\C_1$};
\node(C0) at (10.8,2.6) {$\C_0$};
\node(WBWT2) at (14.4,6) {$\WBWT_2$};
\node[fill=none] at (14.4,5.3) {\tiny (c)};
%\node(CNSORT) at (10.8,14.6) {$\C_\IN*\overline{\C_\IN}\equivW\C_\IN*\SORT\equivW\C_\IN*\T\C_\IN\equivW\C_\IN*\LPO'$};
\node(KN'KN') at (10.8,14.6) {$\K_\IN'*\K_\IN'$};

\draw [->] (TCNN) edge (!CNN);
\draw [->] (!CNN) edge (CNN);
\draw [->] (CNN) edge (lim');
\draw [->] (lim') edge  (lim) ;
\draw [->] (lim') edge (!CN');
\draw [->] (!CN') edge (CN');
\draw [->] (KN'KN') edge (TCN);
\draw [->] (TCN) edge (!CN);
\draw [->] (lim) edge (!L);
\draw [->] (lim) edge (SORT);
\draw [->] (SORT) edge node[left=0.4cm,fill=none] {\tiny (b)} (!CN);
\draw [->] (CN') edge (KN'KN');
\draw [->] (KN'KN') edge (SORT);
\draw [->] (SORT) edge (IVT);
\draw [->] (!L) edge (L);
\draw [->] (L) edge (CR);
\draw [->] (!L) edge (!CR);
\draw [->] (!CR) edge (CR);
\draw [->] (CR) edge (WKL);
\draw [->] (!CR) edge (!PCR);
\draw [->] (CR) edge (PCR);
\draw [->] (!PCR) edge (PCR);
\draw [->] (!PCR) edge (!CN);
%\draw [->] (PCR) edge (WWKL);
%\draw [->] (!PCR) edge (!WWKL);
\draw [->] (PCR) edge (CN);
\draw [->] (WKL) edge (!WWKL);
\draw [->] (!WWKL) edge (WWKL);
\draw [->] (WKL) edge (IVT);
\draw [->] (IVT) edge (!PCC);
\draw [->] (PCC) edge (KN);
\draw [->] (!CN) edge (CN);
%\draw [->] (CN) edge (KN);
\draw [->]  (KN) edge (C2);
\draw [->] (C2) edge (C1);
\draw [->] (C1) edge (C0);
\draw [->] (CN) edge (LPO);
\draw [->] (LPO) edge (C2);
\draw [->] (LPOS) edge (LPO);
\draw [->] (KN'KN') edge (LPOS);
\draw [->] (LPOS) edge (WBWT2);
\draw [->] (!CN) edge (WBWT2);
\draw [->] (WBWT2) edge (C1);
%\draw [->] (!WWKL) edge (!PCC);
%\draw [->] (WWKL) edge (PCC);
\draw [->] (!PCC) edge (PCC);
%\draw [->] (!CN) edge (!PCC);
%\draw [->] (CN) edge (PCC);

 \draw [->,looseness=1] (CN) to [out=270,in=30] (PCC);
 \draw [->,looseness=1] (!CN) to [out=210,in=90] (!PCC);
 \draw [->,looseness=1] (WWKL) to [out=270,in=180] (PCC);
 \draw [->,looseness=1] (!WWKL) to [out=0,in=90] (!PCC);
 \draw [->,looseness=1] (PCR) to [out=270,in=180] (WWKL);
 \draw [->,looseness=1] (!PCR) to [out=0,in=90] (!WWKL);
 
\draw [rounded corners=10pt,dashed]  (0.6,20.1) rectangle node[left=0.6cm,fill=none] {\tiny (a)} (2.2,17.75) ;
\draw [rounded corners=10pt,dashed]  (0.8,12.3) rectangle node[left=0.5cm,fill=none] {\tiny (c)}  (2,10.3);
\draw [rounded corners=10pt,dashed]  (2.5,12.45) rectangle  (3.9,10.25);
\draw [rounded corners=10pt,dashed]  (10.1,4.1) rectangle  (11.5,2.1);
\draw [rounded corners=10pt,dashed]  (4.4,12.4) rectangle  (6,10.2);
\draw [rounded corners=10pt,dashed] (10.1,12.5) rectangle node[right=0.6cm,fill=none] {\tiny (a)} (11.5,10.2);
\draw [rounded corners=10pt,dashed] (10,17.6) rectangle (11.6,15.4);
\draw [rounded corners=10pt,dashed] (6.2,10) rectangle (7.8,7.8);
\draw [rounded corners=10pt,dashed] (8,7.8) rectangle (10.1,5.4);
\end{tikzpicture}